\documentclass[a4paper,12pt]{amsart}

\usepackage{latexsym}
\usepackage{amssymb}
\usepackage{graphicx}
\usepackage{wrapfig}
\usepackage{amsthm}
\usepackage{float}
\usepackage{epsfig}
\usepackage{multirow}
\usepackage{caption}
\usepackage[toc,page]{appendix}
\usepackage{hyperref}
\RequirePackage{color}
\usepackage{tikz-cd}
\usepackage{amscd,enumerate}
\usepackage{amsfonts}

\definecolor{mypink1}{rgb}{0.158, 0.488, 0.178}
\definecolor{suput}{rgb}{0.058, 0.488, 0.08}

\input xy
\xyoption{all}

\usepackage{multicol}

\hypersetup{ colorlinks,
linkcolor=blue,
filecolor=green,
urlcolor=blue,
citecolor=blue }

\input xygraph

\usepackage{calrsfs}
\DeclareMathAlphabet{\pazocal}{OMS}{zplm}{m}{n}

\newtheorem{lemma}{Lemma}[section]
\newtheorem{corollary}[lemma]{Corollary}
\newtheorem{theorem}[lemma]{Theorem}
\newtheorem{proposition}[lemma]{Proposition}
\newtheorem{remark}[lemma]{Remark}
\newtheorem{definition}[lemma]{Definition}

\newtheorem{example}[lemma]{Example}

\newtheorem{notation}[lemma]{Notation}
\newtheorem{corolary}[lemma]{Corollary}

\def\a{\alpha}
\def\b{\beta}

\def\com{\mathop{\text{Com}}}

\newcommand{\U}{{\mathcal{U}}}
\newcommand{\CU}{{\mathcal{CU}}}
\newcommand{\Ker}{{\rm{Ker}}}

\newcommand{\N}{{\mathbb{N}}}
\newcommand{\Z}{{\mathbb{Z}}}
\newcommand{\K}{{\mathbb{K}}}
\newcommand{\Q}{{\mathbb{Q}}}

\newcommand{\uloopr}[1]{\ar@'{@+{[0,0]+(-4,5)}@+{[0,0]+(0,10)}@+{[0,0] +(4,5)}}^{#1}}

\def\aut{\mathop{\text{Aut}}}
\def\der{\mathop{\text{Der}}}
\def\ang#1{\langle #1\rangle}
\definecolor{turquoise2}{rgb}{0,0.898039,0.933333}
\definecolor{magenta}{rgb}{1,0,1}
\definecolor{olivedrab}{rgb}{0.419608,0.556863,0.137255}
\definecolor{purple2}{rgb}{0.568627,0.172549,0.933333}
\definecolor{amethyst}{rgb}{0.6, 0.4, 0.8}
\definecolor{ao(english)}{rgb}{0.0, 0.5, 0.0}
\definecolor{atomictangerine}{rgb}{1.0, 0.6, 0.4}
\definecolor{amber(sae/ece)}{rgb}{1.0, 0.49, 0.0}
\definecolor{alizarin}{rgb}{0.82, 0.1, 0.26}
\definecolor{auburn}{rgb}{0.43, 0.21, 0.1}
\definecolor{aqua}{rgb}{0.0, 1.0, 1.0}

\begin{document}

\subjclass[2010] {17D92, 17A60, 17A36} \keywords{evolution algebra, classification via squares, automorphisms, derivations, associative representations}

\title[Squares and associative representations in evolution algebras]{Squares and associative representations of two dimensional evolution algebras.}
%\title[Evolution algebras of dimension two]{Evolution algebras of dimension two: Their squares, classification, automorphisms, derivations and associative representations.}

\author[M. Gon\c calves]{Maria Inez Cardoso Gon\c calves}
\address{M. I. Cardoso Gon\c calves: Departamento de Matem\'atica, Universidade Federal de Santa Catarina, Florian\'opolis, SC, 88040-900 - Brazil}
\email{maria.inez@ufsc.br}

\author[D. Gon\c calves]{Daniel Gon\c calves}
\address{D. Gon\c calves: Departamento de Matem\'atica, Universidade Federal de Santa Catarina, Florian\'opolis, SC, 88040-900 - Brazil}
\email{daemig@gmail.com}

\author[D. Mart\'\i n]{Dolores Mart\'\i n Barquero}
\address{D. Mart\'\i n Barquero: Departamento de Matem\'atica Aplicada, Escuela de Ingenier\'\i as, Universidad de M\'alaga, Campus de Teatinos s/n. 29071 M\'alaga.   Spain.}
\email{dmartin@uma.es}

\author[C. Mart\'\i n]{C\'andido Mart\'\i n Gonz\'alez}
\address{C\'andido Mart\'\i n Gonz\'alez: Departamento de \'Algebra Geometr\'{\i}a y Topolog\'{\i}a, Fa\-cultad de Ciencias, Universidad de M\'alaga, Campus de Teatinos s/n. 29071 M\'alaga.   Spain.}
\email{candido\_m@uma.es}

\author[M. Siles Molina]{Mercedes Siles Molina}
\address{M. Siles Molina: Departamento de \'Algebra Geometr\'{\i}a y Topolog\'{\i}a, Fa\-cultad de Ciencias, Universidad de M\'alaga, Campus de Teatinos s/n. 29071 M\'alaga.   Spain.}
\email{msilesm@uma.es}

\begin{abstract}
We associate an square to any two dimensional evolution algebra. This geometric object is uniquely determined, does not depend on the basis  and describes the structure and the behaviour of the algebra. We determine the identities of degrees at most four, as well as derivations and automorphisms. We look at the group of automorphisms as an algebraic group, getting in this form a new algebraic invariant. The study of associative representations of evolution algebras is also started and we get faithful representations for most two-dimensional evolution algebras. In some cases we prove that faithful commutative and associative representations do not exist, giving raise to the class of what could be termed as \lq\lq exceptional\rq\rq\ evolution algebras (in the sense of not admitting a monomorphism to an associative algebra with deformed product).
\end{abstract}

\maketitle

%%%%%%%%%%%%%%%%%%%%%%%%%%%%%%
%%%%%%%%%%%%%%%%%%%%%%%%%%%%%%
\section{Introduction}
%%%%%%%%%%%%%%%%%%%%%%%%%%%%%%
%%%%%%%%%%%%%%%%%%%%%%%%%%%%%%

Evolution algebras have connections to dynamical systems, graph theory, Markov process and others, as pointed out by Tian in his book \cite{Tian}. In particular, two dimensional evolution algebras are used as a model for the self-reproduction of alleles in non-Mendelian Genetics (see \cite{VT}). It is in this paper that evolution algebras first appear in 2006.

Due to its importance, it is interesting to understand in deep the mathematical nature of evolution algebras and the use of visual representations of these algebras can help. This was one of the reasons to associate a graph to any evolution algebra and conversely (as done in \cite{Tian, EL}). However, there is not a unique graph for each evolution algebra, even for the two-dimensional ones, as a change of basis can produce graphs of a very different nature. The unicity holds only for perfect evolution algebras (see \cite{CSV1}). This has been the motivation for our search and finding of a geometric (and combinatorial) object, which we call ``square" and  determines in a univocal way a two-dimensional evolution algebra.  An square is an equivalence class (or a union of them) whose combinatorial objects are equivalent to graphs. In this form there is a uniquely defined square associated to a two dimensional evolution algebra.  We hope that such an interplay will grow in a similar way to the celebrated exchange of information between Leavitt path algebras and their associated graphs (see \cite{AAS}).

To illustrate the use of the squares we prove that simplicity of a two dimensional evolution algebra is equivalent to a geometric property of the associated square: the top and the bottom edges appear (Proposition \ref{Oporto}). We also recover the classification of two-dimensional evolution algebras, done in \cite{evo9} for the complex field, in \cite{YC} over an arbitrary field with mild conditions, and in \cite{FFN} without  restriction on the characteristic of the based field, paper that we discovered after we completed the classification. Our approach is very different. The method  used in \cite{FFN} relay on computer aided calculations while ours is based on the squares and does not use a computer.

The dimension of the (algebraic) group of automorphisms of a finite dimensional algebra is an invariant. Thus, if we visualize the group of automorphisms of an evolution algebra as an algebraic group, we get more information than looking at it merely as an abstract group. In the perfect case the algebraic group of a finite dimensional evolution algebra has zero dimension (as can be derived from \cite[Theorem 4.4]{EL}). For two-dimensional non-perfect evolution algebras their automorphisms groups have nonzero dimension, with only one exception (see Table \ref{automorphisms_dev}). \textcolor{magenta}{}

Further to the classification we completely describe the group of automorphisms and the derivations of a two dimensional evolution algebra. The study of the "group" of automorphisms from the viewpoint 
of algebraic groups is much richer than the mere study of automorphism groups as abstract groups. For instance, dealing with the algebraic group of automorphism we can speak of dimension of the group, which is an invariant. Also we could study connectedness or other topological properties which are not at hand in the theory of (abstract) groups.  
We find that the Hopf algebra representing the algebraic group of automorphisms of perfect evolution algebras (in dimension $2$) is \'etale except in one case. This fact implies the shortage of derivations in the perfect case.

We end the paper with a study of associative representations of two dimensional evolution algebras (the concept of associative representations of a nonassociative algebra was recently introduced by Kornev and Shestakov in \cite{KS}). We characterize the universal representation of an evolution algebra of dimension $n$ as a quotient of the free $\ast$ associative algebra in $n$ variables and, for two dimensional evolution algebras, we study when the universal representation is faithful. Furthermore, we introduce the concept of commutative, associative universal faithful representation and describe such algebra (when it exists) for all evolution algebras of dimension two.

The paper is organized as follows: In Section~2 we introduce the square of an evolution algebra and we use then in Section~3 to classify evolution algebras of dimension 2. We describe identities in evolution algebras of dimension 2 in Section~4 and describe the group of automorphisms (as an algebraic group) and the derivations of an evolution algebra of dimension 2 in Section~5. Finally, in Section~6 we study associative representations of two dimensional evolution algebras. We have included at the end of the paper a table containing a summary of notations for the different isomorphic classes with their multiplication tables relative to specific natural basis, and an appendix which contains the codes used in the study of the associative representations.

Before we proceed we recall a few key concepts regarding evolution algebras and set up notation.

In this paper $\K$ we will stand for an arbitrary field. The notation $\K^\times$ will stand for $\K\setminus \{0\}$.

An \emph{evolution algebra} $A$ over a field $\K$  is a $\K$-algebra  provided with a basis $B = \{e_i\}_{i \in I}$
such that $e_ie_j = 0$ whenever $i\ne j$. Such a basis $B$ is called a \emph{natural basis}.

For an evolution algebra $A$,  following \cite[Proposition 2.30]{CSV1}, we know that $A$  is \emph{nondegenerate} (it does not have nonzero ideal $I$,  such that $I^2=0$) if and only if it has a natural basis $B =\{e_1,\cdots ,e_n\}$,  such that $e_i^2\neq 0$ for $1\leq i \leq n$. This is equivalent to say that the same happens for every natural basis.
We will say that $A$ is \emph{perfect} if $A^2=A$.

If we have a $\K$-vector space $V$ and two basis $B$ and $B'$ we will denote by $P_{BB'}$ the \emph{change of basis matrix}, corresponding to a matrix whose columns are the coordinates of the vectors of the basis $B$ relative to $B'$. 

Assume that $B$ and $B'$ are natural basis of a finite dimensional evolution algebra $A$. The relationship between the structure matrices relative to $B$ and $B'$ is given in \cite[Theorem 2.2]{CSV}; concretely,
$$M_{B'}=P_{BB'}M_BP^{(2)}_{B'B},$$
where for a matrix $Q=(q_{ij})$,  $Q^{(2)}=(q^2_{ij})$ and $M_B=(\omega_{ij})$ is the structure matrix relative to the basis $B$, i.e., $e_j^2=\sum_{i}\omega_{ij}e_{i}$.

%%%%%%%%%%%%%%%%%%%%%%%%%%%%%%
%%%%%%%%%%%%%%%%%%%%%%%%%%%%%%
\section{The square associated to a two dimensional evolution algebra}
%%%%%%%%%%%%%%%%%%%%%%%%%%%%%%
%%%%%%%%%%%%%%%%%%%%%%%%%%%%%%

In this section we associate a square to any two dimensional evolution algebra $A$.
%.
Assume that $B=\{e_1, e_2\}$ is a natural basis for $A$ and that the structure matrix is $M_B=(\omega_{ij})$. We associate the diagram that follows.

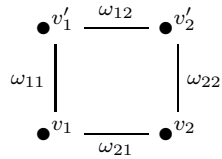
\begin{equation}\tag{S}\label{S}
\centerline{
    $$ \quad {\xymatrix{
\bullet^{v'_1} \ar@{-}[r]^{\omega_{12}} \ar@{-}[d]_{\omega_{11}}&   \bullet^{v'_2}\ar@{-}[d]^{\omega_{22}} \\
\bullet^{v_1}   \ar@{-}[r]_{\omega_{21}}       &         \bullet^{v_2} 
}}$$}
\end{equation}

When $\omega_{ij}=0$, the corresponding side does not appear.
For instance, for an algebra with structure matrix $M_B=\begin{pmatrix} 1 & 0 \\ 1 & 1 \end{pmatrix}$, we get the following diagram:

$$
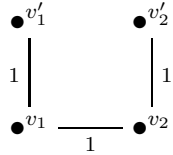
{\xymatrix{
\bullet^{v'_1}  \ar@{-}[d]_{1}&   \bullet^{v'_2}\ar@{-}[d]^{1} \\
\bullet^{v_1}   \ar@{-}[r]_1       &         \bullet^{v_2} 
}}$$

In order to simplify the notation we will write neither the names of the vertices nor the $\omega_{ij}'s$. This diagram will be called a \emph{pseudo-square}.
As with graphs, a change of basis may change the pseudo-square of an algebra. To fix this we make the following definition.

\begin{definition} 
\rm
Let $A$ be an evolution algebra. The \emph{square of} A is the set of all pseudo-squares of $A$. The square of $A$ is obtained by considering all possible natural basis of $A$ and computing the associated pseudo-squares. We denote the square of $A$ by $\mathfrak{S}(A)$.
\end{definition}

As we will see, for perfect evolution algebras each square is a set with only two elements: One pseudo-square and its rotation by an angle of $\pi$. In the non-perfect case more diverse phenomena appear: There are two squares with four elements each. In any case, there is only one square associated to any two dimensional evolution algebra. Conversely, given a square S we associate to it a family of two dimensional evolution algebras (all the algebras whose pseudo-squares belong to S).

In order to better describe the squares associated to an evolution algebra we introduce the following notation. Let $D_2$ be the set of all pseudo-squares arising from algebras of dimension 2 ($D_2$ has cardinal 16). Notice that we have an action of $\Z_2=\{\pm 1\}$ on $D_2$ given by rotation by $\pi$ that induces an (orbit) equivalence in $D_2$. There are 10 equivalence classes that we describe in the table that follows. Denote by $\mathcal D_2$ the set of these equivalence classes.
\begin{center}
\begin{table}[!ht]
\scalebox{0.8}{
\begin{tabular}{|c|cc||cc|c|}
\hline
$\mathfrak{D}_0$ & ${\xymatrix{
\bullet &   \bullet \\
\bullet         &         \bullet
}}$ & 
 & ${\xymatrix{
\bullet\ &   \bullet \ar@{-}[d]\\
\bullet  \ar@{-}[r]       &         \bullet
}}$ & ${\xymatrix{
\bullet\ar@{-}[d] &   \bullet\ar@{-}[l] \\
\bullet         &         \bullet
}}$ & $\mathfrak{D}_8$\\ 
%%%%
\hline
$\mathfrak{D}_7$ & ${\xymatrix{
\bullet \ar@{-}[d]&   \bullet \\
\bullet         &         \bullet
}}$ & ${\xymatrix{
\bullet&   \bullet\ar@{-}[d] \\
\bullet        &         \bullet
}}$ 
& ${\xymatrix{
\bullet\ar@{-}[d]  & \bullet
\\
\bullet  \ar@{-}[r]       &         \bullet
}}$ & ${\xymatrix{
\bullet\ar@{-}[r] &   \bullet\ar@{-}[d] \\
\bullet         &         \bullet
}}$ & $\mathfrak{D}_9$\\ 
\hline
$\mathfrak{D}_6$ & ${\xymatrix{
\bullet&   \bullet\\
\bullet  \ar@{-}[r]       &         \bullet
}}$ & ${\xymatrix{
\bullet\ar@{-}[r] &   \bullet \\
\bullet        &         \bullet
}}$ & ${\xymatrix{
\bullet\ar@{-}[r] \ar@{-}[d]&   \bullet\ar@{-}[d] \\
\bullet        &         \bullet
}}$ & ${\xymatrix{
\bullet \ar@{-}[d]&   \bullet\ar@{-}[d] \\
\bullet  \ar@{-}[r]       &         \bullet
}}$ & $\mathfrak{D}_3$ \\ 
\hline
$\mathfrak{D}_1$ & ${\xymatrix{
\bullet \ar@{-}[d]&   \bullet\ar@{-}[d] \\
\bullet       &         \bullet
}}$ & 
 & ${\xymatrix{
\bullet\ar@{-}[r] \ar@{-}[d]&   \bullet \\
\bullet  \ar@{-}[r]       &         \bullet
}}$ & ${\xymatrix{
\bullet\ar@{-}[r] &   \bullet\ar@{-}[d] \\
\bullet  \ar@{-}[r]       &         \bullet
}}$ & $\mathfrak{D}_4$\\ 
\hline
$\mathfrak{D}_2$ & ${\xymatrix{
\bullet\ar@{-}[r] &   \bullet \\
\bullet  \ar@{-}[r]       &         \bullet
}}$ & 
& ${\xymatrix{
\bullet\ar@{-}[r] \ar@{-}[d]&   \bullet\ar@{-}[d] \\
\bullet  \ar@{-}[r]       &         \bullet
}}$ & & $\mathfrak{D}_5$ \\ 
\hline
\end{tabular}}
\medskip
\caption{Pseudo-squares and the classes by the $\Z_2$ action.}\label{DTDEA}
\end{table}
\end{center}

Next we prove that if $T$ is a pseudo-square in the square of $A$ then its rotation by $\pi$ is also in the square of $A$.

\begin{lemma}\label{isomrotationpi}
The diagrams \eqref{S} and $$ \quad {\xymatrix{
\bullet^{v'_1} \ar@{-}[r]^{\omega_{21}} \ar@{-}[d]_{\omega_{22}}&   \bullet^{v'_2}\ar@{-}[d]^{\omega_{11}} \\
\bullet^{v_1}   \ar@{-}[r]_{\omega_{12}}       &         \bullet^{v_2} 
}}$$
which have been obtained by the action of $\Z_2$ over ${D}_2$, correspond to isomorphic evolution algebras. 
\end{lemma}
\begin{proof}
Let $\{e_1, e_2\}$ and $\{f_1, f_2\}$ be natural basis of two evolution algebras $A_1$, $A_2$ whose associated diagrams are the two in the statement. The map $\varphi: A_1 \to A_2$ given by $\varphi(e_i)=f_j$, for $i, j \in \{1,2\}$ with $i\neq j$, is an isomorphism between both algebras. 
\end{proof}

Below we use the squares to characterize the simplicity of two dimensional evolution algebras.

\begin{proposition}\label{Oporto}
Let $A$ be a two dimensional evolution algebra.
\begin{enumerate}[\rm (i)]
    \item If $A^2=A$, then $A$ is simple if the  top and the bottom edges in every pseudo-square of $\mathfrak{S}(A)$ appear.
    \item If $A$ is simple, then the  top and the bottom edges in every  pseudo-square of $\mathfrak{S}(A)$ appear.
\end{enumerate}
\end{proposition}
\begin{proof}
Let $\{e_1, e_2\}$ be a natural basis of $A$ and let $e_1^2=\alpha e_1 + \beta e_2$; $e_2^2=\gamma e_1 + \mu e_2$.
(i) We are assuming $\beta, \gamma \neq 0$. Since $A^2=A$, we also know that $\alpha\mu-\gamma\beta\neq 0$. Let $I$ be a nonzero ideal generated by an element $ae_1+be_2$, with $a, b\in \K$. If $a=0$ or $b=0$ it is easy to see that $I=A$, so, suppose $a, b\neq 0$. Multiplying $ae_1+be_2$ by $e_1$ we get that $\alpha e_1+\beta e_2\in I$. Now, multiplying $ae_1+be_2$ by $e_2$ we obtain $\gamma e_1 + \mu e_2 \in I$. Since the two vectors $\alpha e_1+\beta e_2$ and $\gamma e_1 + \mu e_2 \in I$ are linearly independent (by the hypothesis), we get $I=A$.

(ii) Note that the simplicity of $A$ implies $A^2=A$, therefore $\alpha \mu -\beta \gamma \neq 0$. We want to show that $\beta$ and $\gamma$ are nonzero. Assume on the contrary that $\b=0$. Then $e_1^2=\a e_1$ implying that $\K e_1$ is an ideal of $A$, a contradiction. Similarly, being $\gamma=0$ contradicts the simplicity of $A$.
\end{proof}

\section{Classification of two dimensional evolution algebras by squares}
We will use the squares we have introduced to classify  two dimensional evolution algebras in full generality, with no restriction over the field or any other.
The squares that appear (and the pseudo-squares they consist of) are the following.

\begin{center}
\begin{table}[!ht]
\scalebox{0.8}{
\begin{tabular}{|c|ccccc|c|}
\hline
$\mathfrak{S}_0$& ${\xymatrix{
\bullet &   \bullet \\
\bullet         &         \bullet
}}$ & 
 \vline ${\xymatrix{
\bullet\ar@{-}[r] \ar@{-}[d]&   \bullet\ar@{-}[d] \\
\bullet  \ar@{-}[r]       &         \bullet
}}$&
 & ${\xymatrix{
\bullet &   \bullet \ar@{-}[d]\\
\bullet  \ar@{-}[r]       &         \bullet
}}$ & ${\xymatrix{
\bullet\ar@{-}[r] \ar@{-}[d]&   \bullet \\
\bullet         &         \bullet
}}$ & $\mathfrak{S}_8$
\\
%%%%
\hline
$\mathfrak{S}_7$&${\xymatrix{
\bullet \ar@{-}[d]&   \bullet \\
\bullet         &         \bullet
}}$ & ${\xymatrix{
\bullet&   \bullet\ar@{-}[d] \\
\bullet        &         \bullet
}}$ 
& 
& ${\xymatrix{
\bullet\ar@{-}[d]&   \bullet \\
\bullet  \ar@{-}[r]       &         \bullet
}}$ & ${\xymatrix{
\bullet\ar@{-}[r] &   \bullet \ar@{-}[d] \\
\bullet         &         \bullet
}}$ & \\ 
\hline
 $\mathfrak{S}_6$&${\xymatrix{
\bullet&   \bullet\\
\bullet  \ar@{-}[r]       &         \bullet
}}$ & ${\xymatrix{
\bullet\ar@{-}[r] &   \bullet \\
\bullet        &         \bullet
}}$ 
&  \vline & ${\xymatrix{
\bullet\ar@{-}[r] \ar@{-}[d]&   \bullet\ar@{-}[d] \\
\bullet        &         \bullet
}}$ & ${\xymatrix{
\bullet \ar@{-}[d]&   \bullet\ar@{-}[d] \\
\bullet  \ar@{-}[r]       &         \bullet
}}$ & $\mathfrak{S}_3$\\ 
\hline
$\mathfrak{S}_1$ &${\xymatrix{
\bullet \ar@{-}[d]&   \bullet\ar@{-}[d] \\
\bullet       &         \bullet
}}$ & &\vline& 
  
 ${\xymatrix{
\bullet\ar@{-}[r] \ar@{-}[d]&   \bullet \\
\bullet  \ar@{-}[r]       &         \bullet
}}$ 
 & ${\xymatrix{
\bullet\ar@{-}[r] &   \bullet\ar@{-}[d] \\
\bullet  \ar@{-}[r]       &         \bullet
}}$ & $\mathfrak{S}_4$\\ 
\hline
$\mathfrak{S}_2$ & ${\xymatrix{
\bullet\ar@{-}[r] &   \bullet \\
\bullet  \ar@{-}[r]       &         \bullet
}}$ & 
& \vline
& ${\xymatrix{
\bullet\ar@{-}[r] \ar@{-}[d]&   \bullet\ar@{-}[d] \\
\bullet  \ar@{-}[r]       &         \bullet
}}$ &  &$\mathfrak{S}_5$\\ 
\hline
\end{tabular}}
\medskip
\caption{Squares for two dimensional evolution algebras.}\label{DTDEAc}
\end{table}
\end{center}

 We start the classification of two dimensional evolution algebras by analyzing the perfect case. 

\subsection{Classification of perfect evolution algebras}

For a two dimensional perfect evolution $\K$-algebra, it is known that 
for any two natural bases $\{v_1,v_2\}$ and $\{v_1',v_2'\}$ there are nonzero scalars $r,s\in\K^\times$ such that $v_1'=r v_1$ and $v_2'= s v_2$ or
$v_1'=r v_2$ and $v_2'= s v_1$, see \cite[Theorem 4.4]{EL}. This implies that the square of $A$ is a uniquely determined element of $\mathcal D_2$. For each of the five elements of $\mathcal D_2$ that  appear in the classification of these algebras we choose one representative. These are the following:
\smallskip

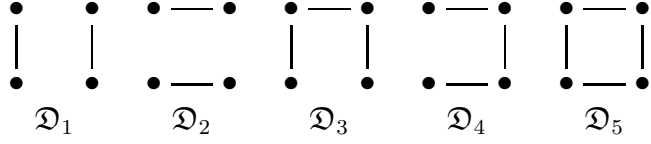
\begin{center}
$\begin{matrix}

\xygraph{
!{<0cm,0cm>;<1cm,0cm>:<0cm,1cm>::}
!{(0,0)}*+{\bullet}="a"
!{(1,0)}*+{\bullet}="b"
!{(0,1)}*+{\bullet}="c"
!{(1,1)}*+{\bullet}="d"
"a"-"c" "b"-"d"} &
\xygraph{
!{<0cm,0cm>;<1cm,0cm>:<0cm,1cm>::}
!{(0,0)}*+{\bullet}="a"
!{(1,0)}*+{\bullet}="b"
!{(0,1)}*+{\bullet}="c"
!{(1,1)}*+{\bullet}="d"
"a"-"b" "c"-"d"} &
\xygraph{
!{<0cm,0cm>;<1cm,0cm>:<0cm,1cm>::}
!{(0,0)}*+{\bullet}="a"
!{(1,0)}*+{\bullet}="b"
!{(0,1)}*+{\bullet}="c"
!{(1,1)}*+{\bullet}="d"
 "b"-"d" "d"-"c" "a"-"c"} &
 \xygraph{
!{<0cm,0cm>;<1cm,0cm>:<0cm,1cm>::}
!{(0,0)}*+{\bullet}="a"
!{(1,0)}*+{\bullet}="b"
!{(0,1)}*+{\bullet}="c"
!{(1,1)}*+{\bullet}="d"
 "b"-"d" "a"-"b" "c"-"d"} &
 \xygraph{
!{<0cm,0cm>;<1cm,0cm>:<0cm,1cm>::}
!{(0,0)}*+{\bullet}="a"
!{(1,0)}*+{\bullet}="b"
!{(0,1)}*+{\bullet}="c"
!{(1,1)}*+{\bullet}="d"
 "a"-"b" "a"-"c" "c"-"d" "d"-"b"}\\
 \mathfrak{D}_1 & \mathfrak{D}_2 & \mathfrak{D}_3 & 
 \mathfrak{D}_4 & \mathfrak{D}_5
\end{matrix}$
\end{center}

By abuse of notation, in what follows, when we speak about a square we will refer to the chosen representative in the class. 

In order to simplify the computations that follow, given an evolution algebra we will make a change of basis so that the entries of its structure matrix are mostly zeros and ones. This is the following lemma.

\begin{lemma}\label{nerja}
Let $A$ be an evolution algebra such that $A^2=A$. Whenever the square associated to $A$ is $\mathfrak{D}_i$, for $i=1, \dots 5$, there is a change of basis matrix such that the structure matrix is one of the following:
\begin{enumerate}[\rm (i)]
\item 
$\begin{pmatrix}
1 & 0 \\ 0 & 1
\end{pmatrix}$ when the square is $\mathfrak{D}_1$. Denote this algebra by $A_1$
\item
$\begin{pmatrix}
0 & \alpha \\ 1 & 0
\end{pmatrix}$, with $\alpha \in \K^\times$, when the square is $\mathfrak{D}_2$. Denote this algebra by $A_{2,\alpha}$.
\item $\begin{pmatrix}
1 & \alpha \\ 0 & 1
\end{pmatrix}$, with $\alpha \in \K^\times$, when the square is $\mathfrak{D}_3$. Denote this algebra by $A_{3,\alpha}$. 
\item  $\begin{pmatrix}
0 & 1 \\ \alpha & 1
\end{pmatrix}$, with $\alpha \in \K^\times$, when the square is $\mathfrak{D}_4$. Denote this algebra by $A_{4,\alpha}$.
\item $\begin{pmatrix}
1 & \alpha \\ \beta & 1
\end{pmatrix}$, with $\alpha, \beta \in \K^\times$, $\a\b\ne 1$ when the square is $\mathfrak{D}_5$. Denote this algebra by $A_{5,\alpha, \beta}$.

\end{enumerate}
\end{lemma}

\begin{proof}
Let $B$ be a natural basis $\{e_1, e_2\}$ such that $e_1^2=\alpha e_1+\beta e_2$ and $e_2^2=\gamma e_1+\delta e_2$. 

(i) Take $B'=\{f_1, f_2\}$ such that $P_{B'B}=
\begin{pmatrix}
\alpha^{-1} & 0 \\ 0 & \delta^{-1}
\end{pmatrix}$.

(ii) Take $B'=\{f_1, f_2\}$ such that $P_{B'B}=
\begin{pmatrix}
1 & 0 \\ 0 & \beta
\end{pmatrix}$ 
and rename $\gamma\beta^2$ as $\alpha$.

(iii) Take $B'=\{f_1, f_2\}$ such that $P_{B'B}=
\begin{pmatrix}
\alpha^{-1} & 0 \\ 0 & \delta^{-1}
\end{pmatrix}$ and rename $\alpha\gamma\delta^{-2}$ as $\alpha$.

(iv) Take $B'=\{f_1, f_2\}$ such that $P_{B'B}=
\begin{pmatrix}
\gamma\delta^{-2} & 0 \\ 0 & \delta^{-1}
\end{pmatrix}$ and rename $\beta\gamma^{2}\delta^{-3}$ as $\alpha$.

(v) Take $B'=\{f_1, f_2\}$ such that $P_{B'B}=
\begin{pmatrix}
\alpha^{-1} & 0 \\ 0 & \delta^{-1}
\end{pmatrix}$ and rename $\gamma\alpha\delta^{-2}$ as $\alpha$ and $\beta\delta\alpha^{-2}$ as $\beta$.
\end{proof}

Next we start the classification of algebras with the same square. This will comprise a series of lemmas, that follow below.

\begin{notation}
\rm
For any $n\in \N$ and any field $\K$, denote the group $\K^\times/(\K^\times)^n$ by $G_n$. For $\rho\in \K^\times$, we write $\overline \rho$ for the class of $\rho$ in $G_n$.

For $n=3$ we define in $G_3$ the relation $\sim$ given by:  $\overline{\a}\sim \overline{\b}$ if and only if $\overline{\a}=\overline{\b}$ or $\overline{\a}=\overline{\b}^2$. This is an equivalence relationship. We denote by $G_3/\sim$ the set of classes and denote by $\widetilde{{\a}}$ each element in this set. Note that, in general, this relation is not associated to any subgroup of $G_3$; in other words, $G_3/\sim$ with the usual operation is not necessarily a group. For an example, consider $G_3$ when $\K$ is the field with four elements. On the contrary, for $G_3$ coming from the field with three elements, $G_3/\sim$ is a group (the trivial one).
\end{notation}

\begin{lemma}\label{caipirinha} Let $A_{2,\alpha}$ be the evolution algebra having a natural basis $\{e_1, e_2\}$ such that $e_1^2=e_2$ and $e_2^2=\alpha e_1$, where $\alpha \in \K^\times$. Then $A_{2,\alpha}$ is isomorphic (as a $\K$-algebra) to $A_{2,\alpha'}$ if and only if $\overline{\alpha} = \overline{\alpha'}$, or $\overline\a=\overline{\a'}^2$, where $\alpha' \in \K^\times$ and classes are considered in $G_3$.
\end{lemma}

\begin{proof}
Let $\{e'_1, e'_2\}$ be a natural basis of $A_{2,\alpha'}$ such that ${e'}_1^2= e'_2$ and ${e'}_2^2=\alpha' e'_1$.

Assume first that $A_{2,\alpha}$ is isomorphic (as a $\K$-algebra) to $A_{2,\alpha'}$, and let $\varphi: A_{2,\alpha} \to A_{2,\alpha'}$ be an isomorphism. By \cite[Corollary 4.7]{EL}, and up to a permutation of the elements of the basis $\{e'_1, e'_2\}$, there exist $r, s\in \K^\times$ such that $\varphi(e_1)= re'_1$ and $\varphi(e_2)= se'_2$. By using the fact that  $\varphi$ is a linear map and that $\varphi(e_i^2)=\varphi(e_i)^2$, a computation gives $\alpha= s^3(r^{-1})^3 \alpha'$, i.e., $\overline{\alpha} = \overline{\alpha'}$.

In case $\varphi(e_1)=re'_2$ and $\varphi(e_2)=se'_1$, a computation as before gives $\overline{\a}=\overline{\a'^2}$.

Assume on the contrary that $\overline{\alpha} = \overline{\alpha'}$. Let $\omega\in \K^\times$ be such that $\omega^3=\alpha'\alpha^{-1}$. Take  $r=\omega$ and $s=r^2$. Then the linear map $\varphi: A_{2, \alpha} \to A_{2,\alpha'}$ such that $\varphi(e_1)= re'_1$ and $\varphi(e_2)= se'_2$ defines an isomorphism between both evolution algebras.

Finally, if $\bar\a=\overline{\a'}^2$, i.e., $\alpha=\alpha’^2k^3$ for some $k\in \K^\times$, then the linear map $\varphi: A_{2, \alpha}\to A_{2, \alpha'}$ given by $\varphi(e_1)=ke'_2$ and $\varphi(e_2)=k^2\alpha'e_1'$ gives an algebra isomorphism. 
\end{proof}

\begin{lemma}\label{saquerinha} Let $A_{3, \alpha}$ be the evolution algebra having a natural basis $\{e_1, e_2\}$ such that $e_1^2= e_1$ and $e_2^2=\alpha e_1 + e_2$, where $\alpha \in \K^\times$. Let $\alpha' \in \K^\times$. Then $A_{3,\alpha}$ is isomorphic (as a $\K$-algebra) to $A_{3,\alpha'}$ if and only if $\alpha = \alpha'$.
\end{lemma}
\begin{proof}
Let $\{e'_1, e'_2\}$ be a natural basis of $A_{3,\alpha'}$ such that ${e'_1}^2= e'_1$ and ${e'}_2^2=\alpha e_1'+ e_2'$.

Assume first that $A_{3,\alpha}$ is isomorphic (as a $\K$-algebra) to $A_{3,\alpha'}$, and let $\varphi: A_{3,\alpha} \to A_{3,\alpha'}$ be an isomorphism. By \cite[Corollary 4.7]{EL}, and up to a permutation of the elements of the basis $\{e'_1, e'_2\}$, there exist $r, s\in \K^\times$ such that $\varphi(e_1)= re'_1$ and $\varphi(e_2)= se'_2$. Since $\varphi$ is a linear map and $\varphi(e_i^2)=\varphi(e_i)^2$, a computation gives  that $\alpha=\alpha'$.

The converse follows trivially.
\end{proof}

\begin{lemma}\label{pina_colada} Let $A_{4, \alpha}$ be the evolution algebra having a natural basis $\{e_1, e_2\}$ such that $e_1^2= \alpha e_1$ and $e_2^2= e_1 + e_2$, where $\alpha \in \K^\times$. Let $\alpha' \in \K^\times$. Then $A_{4,\alpha}$ is isomorphic (as a $\K$-algebra) to $A_{4,\alpha'}$ if and only if $\alpha = \alpha'$.
\end{lemma}
\begin{proof}
Let $\{e'_1, e'_2\}$ be a natural basis of $A_{4,\alpha'}$ such that ${e'}_1^2= \alpha' e'_1$ and ${e'}_2^2= e_1'+ e_2'$.
Assume first that $A_{4,\alpha}$ is isomorphic (as a $\K$-algebra) to $A_{4,\alpha'}$, and let $\varphi: A_{4,\alpha} \to A_{4,\alpha'}$ be an isomorphism. By \cite[Corollary 4.7]{EL}, and up to a permutation of the elements of the basis $\{e'_1, e'_2\}$, there exist $r, s\in \K^\times$ such that $\varphi(e_1)= re'_1$ and $\varphi(e_2)= se'_2$. Since $\varphi$ is a linear map and $\varphi(e_i^2)=\varphi(e_i)^2$, a computation gives  that $\alpha=\alpha'$.
The converse follows trivially.
\end{proof}

\begin{lemma}\label{guarana} Let $A_{5,\alpha, \beta}$ be the evolution algebra having a natural basis $\{e_1, e_2\}$ such that $e_1^2= e_1+\beta e_2$ and $e_2^2=\alpha e_1 + e_2$, where $\alpha, \beta \in \K^\times$. Let $\alpha', \beta' \in \K^\times$. Then $A_{5,\alpha, \beta}$ is isomorphic (as a $\K$-algebra) to $A_{5,\alpha', \beta'}$ if and only if $(\alpha, \beta) = (\alpha', \beta')$ or $(\alpha, \beta) = (\beta',\alpha')$.
\end{lemma}
\begin{proof}
Take a natural basis $\{e'_1, e'_2\}$ of $A_{5,\alpha',\beta'}$ such that ${e'}_1^2= e'_1 + \beta' e_2'$ and ${e'}_2^2=\alpha' e_1'+ e_2'$.
Assume that $A_{5,\alpha,\beta}$ is isomorphic (as a $\K$-algebra) to $A_{5,\alpha',\beta'}$, and let $\varphi: A_{5,\alpha,\beta} \to A_{5,\alpha',\beta'}$ be an isomorphism. By \cite[Corollary 4.7]{EL}, and up to a permutation of the elements of the basis $\{e'_1, e'_2\}$, there exist $r, s\in \K^\times$ such that $\varphi(e_1)= re'_1$ and $\varphi(e_2)= se'_2$. Since $\varphi$ is a linear map and $\varphi(e_i^2)=\varphi(e_i)^2$, a computation gives  that $(\alpha, \beta) = (\alpha', \beta')$ or $(\alpha, \beta) = (\beta',\alpha')$.

On the other hand, if  $(\alpha, \beta) = (\alpha', \beta')$, then the linear map $\varphi: A_{5,\alpha,\beta} \to A_{5,\alpha',\beta'}$ such that $\varphi(e_1)= e'_1$ and $\varphi(e_2)= e'_2$ defines an isomorphism between both evolution algebras.  And if  $(\alpha, \beta) = (\beta',\alpha')$, then the linear map $\varphi: A_{5,\alpha,\beta} \to A_{5,\alpha',\beta'}$ such that $\varphi(e_1)= e'_2$ and $\varphi(e_2)= e'_1$ defines an isomorphism between both evolution algebras.
\end{proof}

We compile the results above in the following classification theorem.

\begin{theorem}\label{AsocDiag}
Let $A$ be a perfect two dimensional evolution algebra over an arbitrary field. The square of $A$ is one of the following:\medskip

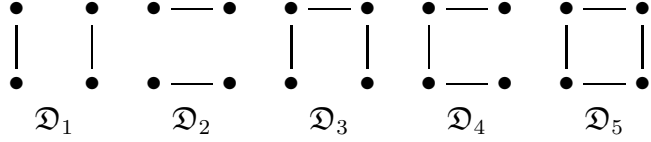
\begin{center}
$\begin{matrix}

\xygraph{
!{<0cm,0cm>;<1cm,0cm>:<0cm,1cm>::}
!{(0,0)}*+{\bullet}="a"
!{(1,0)}*+{\bullet}="b"
!{(0,1)}*+{\bullet}="c"
!{(1,1)}*+{\bullet}="d"
"a"-"c" "b"-"d"} &
\xygraph{
!{<0cm,0cm>;<1cm,0cm>:<0cm,1cm>::}
!{(0,0)}*+{\bullet}="a"
!{(1,0)}*+{\bullet}="b"
!{(0,1)}*+{\bullet}="c"
!{(1,1)}*+{\bullet}="d"
"a"-"b" "c"-"d"} &
\xygraph{
!{<0cm,0cm>;<1cm,0cm>:<0cm,1cm>::}
!{(0,0)}*+{\bullet}="a"
!{(1,0)}*+{\bullet}="b"
!{(0,1)}*+{\bullet}="c"
!{(1,1)}*+{\bullet}="d"
 "b"-"d" "d"-"c" "a"-"c"} &
 \xygraph{
!{<0cm,0cm>;<1cm,0cm>:<0cm,1cm>::}
!{(0,0)}*+{\bullet}="a"
!{(1,0)}*+{\bullet}="b"
!{(0,1)}*+{\bullet}="c"
!{(1,1)}*+{\bullet}="d"
 "a"-"b" "a"-"c" "c"-"d"} &
 \xygraph{
!{<0cm,0cm>;<1cm,0cm>:<0cm,1cm>::}
!{(0,0)}*+{\bullet}="a"
!{(1,0)}*+{\bullet}="b"
!{(0,1)}*+{\bullet}="c"
!{(1,1)}*+{\bullet}="d"
 "a"-"b" "a"-"c" "c"-"d" "d"-"b"}\\
 \mathfrak{D}_1 & \mathfrak{D}_2 & \mathfrak{D}_3 & 
 \mathfrak{D}_4 & \mathfrak{D}_5
\end{matrix}$
\end{center}
And, up to isomorphism:
\begin{enumerate}[\rm(i)]
    \item There is only one evolution algebra, $A_1$, associated to $\mathfrak{D}_1$. 
    \item For each element $\widetilde{\alpha}\in G_3/\sim$, there is only one evolution algebra, $A_{2,\alpha}$ associated to $\mathfrak{D}_2$.
    \item For each element $\alpha \in \K^\times$, there is only one evolution algebra, $A_{3,\alpha}$ associated to $\mathfrak{D}_3$. 
    \item For each element $\alpha \in \K^\times$, there is only one evolution algebra, $A_{4,\alpha}$ associated to $\mathfrak{D}_4$.  
    \item For each pair $(\alpha, \beta)\in \K^\times\times \K^\times$, such that $\alpha\beta\neq 1$, there is only one evolution algebra (up to permutation of the components of the pair), $A_{5,\alpha,\beta}$ associated to $\mathfrak{D}_5$.  
\end{enumerate}
\end{theorem}
\begin{proof}
Item (i) follows from Lemma~\ref{nerja}; (ii) by Lemmas~\ref{nerja} and \ref{caipirinha} (iii) by Lemmas~\ref{nerja} and \ref{saquerinha}; (iv) by Lemmas~\ref{nerja} and \ref{pina_colada} and (v) by Lemmas~\ref{nerja} and \ref{guarana}.
\end{proof}

\begin{remark}
\rm
In Theorem \ref{AsocDiag} the algebra $A_1$ is associative and the only one (perfect) which is decomposable, i.e. $A=I_1\oplus I_2$, where $I_i$ is the ideal generated by $e_i$ and $I_i^2\neq0$, for $\{e_1, e_2\}$ a natural basis having structure matrix the one given in Lemma \ref{nerja}. More concretely, it is isomorphic to $\K \times \K$. 
\end{remark}

Finally we describe the ideal structure of the perfect two dimensional evolution algebras.

\begin{proposition}\label{IdealesPerfecto}
Let $A$ be a perfect two dimensional evolution algebra. Then $A$ has a one  dimensional (evolution) ideal only in case $A_{3,\a}$ and has two one  dimensional ideals only in case $A_1$. 

In terms of squares, $A$ has a one  dimensional ideal if a pseudo-square in $\mathfrak{S}(A)$ has the left line but not the bottom one, or has the right line but not the top one. Each of this cases determines a one  dimensional ideal, generated by the corresponding vector in the natural basis relative to which the pseudo-square is given.
\end{proposition}

\subsection{The non-perfect case}
The case $A^2=0$, which corresponds to square $\mathfrak{D}_0$, is trivial. Now, 
 we will investigate the squares associated to a two dimensional evolution $\K$-algebra $A$ such that $0\ne A^2\subsetneq A$, and the classification of such algebras.

When $A^2\ne A$ the situation is more cumbersome than in the perfect case as we will see: There are squares that consist of union of two equivalence classes in $D_2$.  In this case, to represent the square we always choose the pseudo-square with less edges. 

Notice that, if $A$ is non perfect and $T$ is a pseudo-square in the square of $A$ then $T$ is not in $\mathfrak{D_1}$ to $\mathfrak{D_4}$. For each of the five elements of $\mathcal D_2$ that appear in the classification of the non-perfect evolution algebras we choose one representative. These are the following:

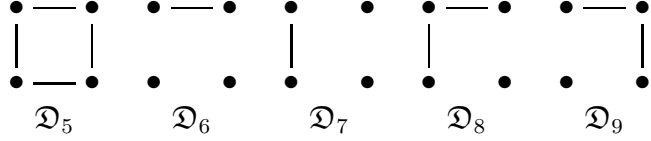
\begin{center}
$\begin{matrix}

\xygraph{
!{<0cm,0cm>;<1cm,0cm>:<0cm,1cm>::}
!{(0,0)}*+{\bullet}="a"
!{(1,0)}*+{\bullet}="b"
!{(0,1)}*+{\bullet}="c"
!{(1,1)}*+{\bullet}="d"
"a"-"b" "a"-"c" "c"-"d" "d"-"b"} &
\xygraph{
!{<0cm,0cm>;<1cm,0cm>:<0cm,1cm>::}
!{(0,0)}*+{\bullet}="a"
!{(1,0)}*+{\bullet}="b"
!{(0,1)}*+{\bullet}="c"
!{(1,1)}*+{\bullet}="d"
"c"-"d"} &
\xygraph{
!{<0cm,0cm>;<1cm,0cm>:<0cm,1cm>::}
!{(0,0)}*+{\bullet}="a"
!{(1,0)}*+{\bullet}="b"
!{(0,1)}*+{\bullet}="c"
!{(1,1)}*+{\bullet}="d"
"a"-"c"} &
 \xygraph{
!{<0cm,0cm>;<1cm,0cm>:<0cm,1cm>::}
!{(0,0)}*+{\bullet}="a"
!{(1,0)}*+{\bullet}="b"
!{(0,1)}*+{\bullet}="c"
!{(1,1)}*+{\bullet}="d"
 "a"-"c" "c"-"d"} &
 \xygraph{
!{<0cm,0cm>;<1cm,0cm>:<0cm,1cm>::}
!{(0,0)}*+{\bullet}="a"
!{(1,0)}*+{\bullet}="b"
!{(0,1)}*+{\bullet}="c"
!{(1,1)}*+{\bullet}="d"
 "b"-"d" "c"-"d"}
\\
 \mathfrak{D}_5 & \mathfrak{D}_6 & \mathfrak{D}_7 & 
 \mathfrak{D}_8 &
 \mathfrak{D}_9
\end{matrix}$
\end{center}

By abuse of notation, in what follows, when we speak about a square, we will refer to the chosen representative in the class.

The classification of the algebras we deal with relies on the powers of the algebra. If $A$ is an evolution algebra, 
$A^2$ is the $\K$-linear span of the set $\{xy\ \vert \ x,y\in A\}$, while $A^3$ is the linear span of $\{xy\ \vert\ x\in A, y\in A^2\}$. Note that $(A^2)^2$, which is the $\K$-linear span of $\{xy\ \vert\ x,y\in A^2\}$, is not necessarily the same as the linear span of $\{x y\ \vert\ x\in A,y\in A^3\}$.

Assume $A^2=\K u$ with $u\ne 0$.  
Then we analyze three possibilities:
\smallskip

{\bf Case 1}: $A^3\neq 0$ and $(A^2)^2=0$. 
\noindent

\noindent
Note that this is equivalent to $u^2=0$ and $Au\neq 0$. 
 Since $Au\neq 0$, after scaling if necessary, there is a nonzero $v\in A$ such that $uv=u$. Then $\{u,v\}$ is a basis of $A$ and
$v^2=k u$. Now, changing $ku$ to $u'$ and taking $v'=v$ we get that $u'^2 = 0$, $u'v'=u'$ and $v'^2=u'$. Take $e=v'$ and $f=-u'+v'$. Then
$e^2= e-f$, $f^2=-e+f$ and $ef=0$, hence the algebra has a pseudo-square in $\mathfrak{D}_5$.

\medskip

It is not difficult to see that for any other natural basis $\{g, h\}$, the associated pseudo-square is the same. Collecting this information we have the following result.

\begin{proposition}\label{cuboNonulo1} Let $A$ be a two dimensional evolution algebra with $\dim(A^2)=1$, $A^3\ne 0$ and $(A^2)^2=0$. Then 
there is a natural basis $\{e,f\}$ with $e^2=e-f$ and $f^2=-e+f$. The square of the algebra is $\mathfrak{D}_5$.
This algebra, that will be called $A_5$, is not alternative (Recall that an algebra $A$ is \emph{alternative} if $x^2y=x(xy)$ and $yx^2 = (yx)x$, for all $x, y\in A$). 
\end{proposition}

\begin{remark}
\rm
Note that by Proposition \ref{cuboNonulo1}, up to isomorphism, there is only one two dimensional non-perfect evolution algebra satisfying $A^3\neq 0$ and $(A^2)^2=0$. This does not mean that the algebra $A$ is nilpotent because $(...((A^2) A)A...)A= \K (e-f)$.
\end{remark}

\smallskip

{\bf Case 2:} $A^3=0$. 
\noindent

\noindent
Note that this is equivalent to say $Au=0$. For any $v\in A$ which is linearly independent to $u$,  it happens that $\{u,v\}$ is a natural basis of $A$ such that $u^2=0$, $v^2= \alpha u\ne 0$.
This implies $\K u=\hbox{Ann}(A)$ is a one  dimensional ideal of $A$ and there is no other proper nontrivial ideal. Indeed, assume $0\ne I:=\K(\beta u+ \gamma v)\triangleleft A$, with $\beta, \gamma\in\K$. Then $I\ni (\beta u+ \gamma v)v=\gamma \alpha u$. If $\gamma=0$ then $I=\K u=\hbox{Ann}(A)$; otherwise
$u\in I$ and hence $\hbox{Ann}(A)=\K u\subseteq I$ so that $I=\hbox{Ann}(A)$. Consequently, $A$ is not decomposable, satisfies $A^3=0$
and $\dim(A^2)=1$. It has a pseudo-square, relative to the basis $\{u,v\}$, in $\mathfrak{D}_6$ and it is not difficult to see that any other natural basis provides the same pseudo-square (up to rotation by $\pi$).

Defining  $u':=\alpha u$ and considering the new basis $\{u',v\}$, we have $u'^2=u'v=0$ and $v^2=u'$. This means that, up to isomorphism, there is only one evolution algebra in this case.
Moreover, it is associative. We collect all the information in the result that follows.
\smallskip
\begin{proposition}\label{Dseis} There is only one isomorphism class of two dimensional evolution algebras $A$ such that  $\dim(A^2)=1$ and $A^3=0$. Its square
is $\mathfrak{D}_6$. This algebra, that we call $A_6$, is associative.
 \end{proposition}

\smallskip

{\bf Case 3:} $A^3\ne 0$ and $(A^2)^2\ne 0$. 

\noindent

\noindent
Note that this is equivalent to $uA\ne 0$ and $u^2\ne 0$.  
Consider the left multiplication operator $L_u\colon A\to A$. Since $A^2=\K u$, there is some linear map $f\colon A\to\K$ such that 
$L_u(x)=f(x)u$ for any $x\in A$. Since $f\ne 0$, we have $\dim(\hbox{Im}(f))=1$. Hence $\dim(\ker(f))=1$ and so there is some nonzero $v$ such
that $uv=0$. Summarizing, $\{u,v\}$ is linearly independent and  we have a natural basis $\{u,v\}$ with 
$u^2=\alpha u$ and $v^2=\beta u$, where $\alpha,\beta\in\K$ with $\alpha\ne 0$. Scaling $u$ if necessary, we may consider without loss of generality
that the basis satisfies $u^2=u$, $v^2=\beta u$.
There are now two possibilities that we will call  Case 3.1 and Case 3.2.

{\bf Case 3.1:}  $\beta=0$.  

\noindent

\noindent
In this case the natural basis is $\{u,v\}$ with $u^2=u$, $v^2=0$. $A$ is decomposable and isomorphic to  $\K\times \K$ with 
multiplication $(x,y)(x',y')=(xx',0)$. We call this algebra $A_7$. It has a pseudo-square in $\mathfrak{D}_7$. Now we see that any other pseudo-square for $A$ has to be in $\mathfrak{D}_7 \cup \mathfrak{D}_9$. Indeed, let $\{e, f\}$ be a natural basis. Since ${\rm Ann}(A)$ is one dimensional (because $v^2=0$), by \cite[Proposition 2.18]{CSV1}, then $e^2=0$ or $f^2=0$. By symmetry, we may assume $e^2=0$. This implies that the left vertical side and the bottom horizontal side in a pseudo-square cannot appear. On the other hand, a pseudo-square in $\mathfrak{D}_6$ cannot appear because, in this case, $f^2f^2=0$, but we are not considering this situation (as $(A^2)^2\neq 0$). This shows the claim.
\medskip

\begin{lemma}\label{orujonorrepe}
The evolution algebras associated to the pseudo-squares that follow
$$
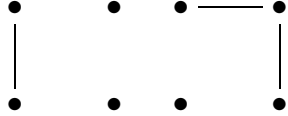
{\xymatrix{
\bullet \ar@{-}[d]&   \bullet \\
\bullet         &         \bullet
}} \quad 
{\xymatrix{
\bullet \ar@{-}[r]&   \bullet \ar@{-}[d]\\
\bullet         &         \bullet
}} $$ 
are isomorphic. 
\end{lemma}
\begin{proof}
Let  $A_1$ be an evolution algebra whose pseudo-square is the first one in the statement, and  $\{e_1, e_2\}$ be a natural basis of $A_1$, i.e., $e_1^2 = \alpha e_1$, for $\alpha \neq 0$ and $e_2^2=0$. Let $A_2$ an evolution algebra whose pseudo-square is the second one in the statement, and   $\{f_1,f_2\}$  be a natural basis of $A_2$, i.e., $f_1^2=0$ and $f_2^2=\gamma f_1 + \delta f_2$, where $\gamma\delta\neq0$. Let $\varphi: A_1 \rightarrow A_2$, be such that $\varphi(e_1)= \alpha\gamma{(\delta^2)}^{-1}f_1+\alpha\delta^{-1}f_2$ and $\varphi(e_2)=f_1$. It is easy to verify that $\varphi$ is an algebra isomorphism between $A_1$ and $A_2$.
\end{proof}

We collect the information above.

\begin{corolary}
There is (up to isomorphism) only one two dimensional evolution algebra $A$ such that $A^3\neq0$, $(A^2)^2\neq 0$ having nontrivial annihilator. Its square is $\mathfrak{D}_7 \cup \mathfrak{D}_9$.
\end{corolary}

{\bf Case 3.2:} We assume $\beta\ne 0$. 
\noindent

\noindent
The natural basis is $\{u,v\}$ with $u^2=u$, $v^2=\beta u$. We denote by $A_{8, \a}$ the algebra with natural basis 
$\{u,v\}$ and $u^2=u$, $v^2=\a u$.

It has a pseudo-square  in $\mathfrak{D}_8$. To find its square, we need the following lemma. 
\medskip

\begin{lemma}\label{queimada}
Let $A$ be a evolution algebra with natural basis $\{u,v\}$ such that $u^2=\alpha u$ and $v^2=\beta u$, with $\alpha, \beta\in \K^\times$ (notice that A is non-perfect). Then, for any other natural basis of $A$ its pseudo-square is one of the following
$$
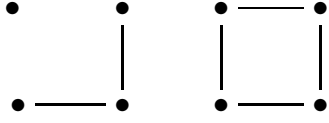
{\xymatrix{
\bullet\ &   \bullet \ar@{-}[d]\\
\bullet  \ar@{-}[r]       &         \bullet
}} \quad \quad
{\xymatrix{
\bullet\ar@{-}[r] \ar@{-}[d]&   \bullet\ar@{-}[d] \\
\bullet  \ar@{-}[r]       &         \bullet
}}$$

\end{lemma}
\begin{proof}

First notice that, without loss of generality, we can assume that $u^2=u$ and $v^2=\beta u$.

Let $\{e,f\}$ be a natural basis of $A$. Then $e=xu + y v$ and $f=x'u + y' v $, with $\Delta=xy'-x'y \neq 0 $ and $xx'+\b yy'=0$. 

If $x'=0$ then $e^2=xe$ and $f^2=\frac{y'^2\b}{x} e$; hence the square is the left one.

If $x'\neq 0$ then $e^2=\frac{x^2+y^2\b}{\Delta}(y'e-yf)$ and $f^2=\frac{x'^2+y'^2\b}{\Delta}(y'e-yf)$, with $y\neq 0$ (notice that $x^2+y^2\b$ and $x'^2+y'^2\b$ are non-zero, since $A$ is nondegenerate). Therefore, if $y'=0$ then the diagram is the one in the left hand side in the statement of the lemma, and if $y'\neq 0$ then the diagram is the one in the right side.
\end{proof}

\begin{corolary}
The square of the algebra $A_{8, \a}$ is  $\mathfrak{D}_5 \cup \mathfrak{D}_8$.
\end{corolary}

\begin{lemma} The evolution algebras $A_{8, \a}$ and $A_{8, \b}$ (for $\a, \b \in \K^\times$) are isomorphic if and only if  $\overline{\a} = \overline{\b}$, where classes are considered in $G_2$. 
\end{lemma}
\begin{proof} Fix a natural basis $\{u,v\}$ in $A_{8,\a}$ such that $u^2=u$, $v^2=\a u$. Fix also a natural basis $\{u',v'\}$ of $A_{8,\b}$ with $u'^2=u'$, $v'^2=\b u'$.

 Assume first that $\bar\a=\bar\b$, that is, $\a\b^{-1}=\tau^2$.
Define $\Phi\colon A_{8,\a}\to A_{8, \b}$ as the linear extension of $\Phi(u)=u'$, $\Phi(v)=\tau v'$. 
To prove that $\Phi$ is an algebra homomorphism if suffices to see that
$\Phi(v^2)=\Phi(v)^2$ (since $u$, $u'$ are idempotents and $uv=u'v'=0$).
Indeed, $\Phi(v^2)=\Phi(\a u)=\a u'$ and $\Phi(v)^2=(\tau v')^2=\tau^2 v'^2=\tau^2\b u'=\a u'$. Thus
$A_{8,\a}\cong A_{8,\b}$. 

Reciprocally assume that there is an isomorphism $\theta\colon A_{8,\a}\to A_{8,\b}$.
Write $\theta(u)=a u'+b v'$ and $\theta(v)=c u'+d v'$ for some $a,b,c,d\in\K$. Then 
$$\theta(u)=\theta(u)^2=(a^2+b^2\b) u', \quad \hbox{which implies}\quad \begin{cases}b=0,\cr a^2+b^2\b=a.\end{cases}$$
So far we have $b=0$, $a=1$ or, equivalently, $\theta(u)=u'$. Also
$$\ 0=\theta(u)\theta(v)=u'(c u'+d v')=c u' \  \hbox{ and hence } c=0.$$
So $\theta(v)=d v'$ and since 
$$\begin{cases}\theta(v^2)=\a\theta(u)=\a u',\cr 
\theta(v)^2=(d v')^2=d^2\b u',\end{cases}$$
 we get $d^2 \b =\a$ so that $\bar\b=\bar\a$ and this finishes the proof of the lemma.
 \end{proof}

 Summarizing the possibilities  we have the result that follows.
  
\begin{proposition}\label{SumaDeTres} If $A$ is a two dimensional evolution algebra with $\dim(A^2)=1$, $A^3\ne 0$ and $(A^2)^2\ne 0$,
 then either $A$ is isomorphic to $A_7$, which is decomposable and isomorphic to $\K\times\K$ with product $(x,y)(x',y')=(xx',0)$, or 
 $A$ is indecomposable and isomorphic to the algebra $A_{8,\a}$ ($\a\ne 0$) 
 with natural basis $\{u,v\}$ and product $u^2=u$, $v^2=\a u$. In the latter case there are so many isomorphism classes as the order of the group $G_2$. More precisely $A_{8,\a}\cong A_{8,\b}$ if and only if
 $\bar\a=\bar \b$ in $G_2$. The associated squares are   $\mathfrak{D}_7 \cup \mathfrak{D}_9$ and $\mathfrak{D}_5\cup \mathfrak{D}_8$,  respectively.
\end{proposition}
\begin{proof}The only thing to be proved is that the algebra $A_{8,\a}$ ($\a\ne 0$) is indecomposable. This follows from the fact that the only one dimensional ideals of this algebra are the (vector) subspaces generated by $au$ for any $a\in \K^\times$.
\end{proof}

\begin{remark}
\rm
For a field in which every element has a square root the group $G_2$ is trivial, hence
in this case all the algebras $A_{8, \a}$ are isomorphic. This is the case for any algebraically closed field. If $\K=\mathbb{R}$, for example, then $G_2\cong\Z_2$ and hence there are two isomorphism classes.
\end{remark}

Now we collect all the information.

\begin{theorem}\label{NonPerfectCase}
Let $A$ be a non-perfect two dimensional evolution algebra over an arbitrary field. The  square of $A$ is one of the following: $\mathfrak{S}_5=\mathfrak{D}_5$, $\mathfrak{S}_6=\mathfrak{D}_6$, $\mathfrak{S}_7=\mathfrak{D}_7 \cup \mathfrak{D}_9$, and $\mathfrak{S}_8=\mathfrak{D}_5 \cup \mathfrak{D}_8$.

And, up to isomorphism:

\begin{enumerate}[\rm(i)]
    \item $A_5$ is the only evolution algebra such that $\mathfrak{S}(A_5)=\mathfrak{S}_5$.
    \item $A_6$ is the only evolution algebra such that $\mathfrak{S}(A_6)=\mathfrak{S}_6$.
     \item $A_7$ is the only evolution algebra such that $\mathfrak{S}(A_7)=\mathfrak{S}_7$.
     \item For each element $\overline{\alpha}\in G_2$, the algebra $A_{8, \alpha}$ is the only evolution algebra such that  $\mathfrak{S}(A_{8, \alpha})=\mathfrak{S}_8$.
    \end{enumerate}
\end{theorem}

To finish, we study the ideal structure  of the non-perfect evolution algebras.

\begin{proposition}\label{IdealesNoPerfecto}
Every linearly independent set of vectors generates a non-zero ideal of $A_0$. $A_7$ has two one  dimensional (evolution) ideals.
Any other non-perfect two dimensional evolution algebra has a unique one  dimensional (evolution) ideal.

In terms of squares, except for $A_5$, any non-perfect two dimensional evolution algebra has a one  dimensional ideal if and only if a pseudo-square with least lines in $\mathfrak{S}(A)$ has not the bottom side, or has  not the top side. Each of this cases determines a one  dimensional ideal, generated by the corresponding vector in the natural basis relative to which the pseudo-square is given.
\end{proposition}

\section{Identities of two dimensional evolution algebras}

In this section we start the study of identities in evolution algebras. Note that since every evolution algebra is commutative, the degree two identities satisfied by evolution algebras are linear multiples of $xy-yx$. 
We  determine the identities of degree three and four in the case of two dimensional evolution algebras.

Let $\com_\K(X)$ denote the (nonassociative) free commutative $\K$-algebra on a set of generators $X$. For a natural number $n$, an element $w\in\com_\K(X)$ is said to be $n$-\emph{linear} if
it is a linear combination of words consisting of $n$ different elements of $X$. Let $C_3$ denote the vector space of $3$-linear identities of $\com_\K(X)$.  Our aim in this section is to answer the question of which elements of $C_3$ are satisfied by the two dimensional evolution algebras. We will use the squares introduced in the previous section.

An arbitrary element
$w:=w(x,y,z)=\lambda _1 (x y) z+\lambda _2 (y z)
   x+\lambda _3 (z x) y\in C_3$, where $x, y, z \in X$ and $\lambda_i \in \K$ (for $i\in \{1, 2,3\}$), is an identity for a two dimensional  evolution algebra $A$ if and only if $w(e_i,e_j,e_k)=0$  for all $i,j,k\in\{1,2\}$, where $\{e_1,e_2\}$ is any natural basis of $A$.
   
We divide the study into two pieces, depending if the evolution algebra is perfect or not. The computations are not included; they have been checked using \cite{Wolfram}.

\subsection{Identities of two dimensional perfect evolution algebras}
\hskip 5.cm

\noindent
We first study the identities of degree 3 satisfied by a two dimensional perfect evolution algebra $A$. If $A$ has square $\mathfrak{D}_1$, the set of equations $w(e_i,e_j,e_k)=0$ provides the solution $\lambda_1+\lambda_2+\lambda_3=0$ and from here, we get the identity $\lambda _1 (x y) z+\lambda _2 (y z)
   x-(\lambda _1+\lambda_2) (z x) y$, so the vector space of solutions is generated by the associative identity (this can be obtained easily by taking first $\lambda_1=1$ and $\lambda_2=0$ and then $\lambda_1=0$ and $\lambda_2=1$). Thus, the only $3$-linear identity satisfied by the algebra $A_1$ is the associative identity.
   
 To study the remaining algebras we use Theorem \ref{AsocDiag} and solve the different linear systems having eight equations and three variables which appear. It happens that the only solution is the trivial. We omit the details and collect the result.

   \begin{proposition}\label{IdGradoTres}
   The only perfect two dimensional evolution algebra satisfying a nontrivial identity of degree $3$ is $A_1$ and this identity is the associative. 
   \end{proposition}
   
  The second step is to consider the space $C_4$ of $4$-linear identities which are satisfied. Using again Theorem \ref{AsocDiag}, it is enough to take into account the algebras associated to $\mathfrak{D}_i$ for $i=1,\ldots,5$.
 Let $w(x,y,z,t)\in C_4$, and write
   $$w(x,y,z,t)=\lambda _1 (x y) (z t)+\lambda _2 (x
   z) (y t)+\lambda _3 (x t) (y
   z)+\lambda _4 (x (y z)) t+\lambda _5
   (y (x z)) t+$$
   $$\lambda _6 (z (y
   x)) t+\lambda _7 (t (x y)) z+\lambda
   _8 (x (t y)) z+\lambda _9 (y (x
   t)) z+\lambda _{10} (y (t z))
   x+$$
   $$\lambda _{11} (t (y z)) x+\lambda _{12}
   (z (t y)) x+\lambda _{13} (x (t
   z)) y+\lambda _{14} (z (t x))
   y+\lambda _{15} (t (x z)) y.$$
   
Then we solve the linear system 
\begin{equation}\label{mistela}
    w(e_i,e_j,e_k,e_l)=0,
\end{equation} 
for $i,j,k,l\in\{1,2\}$, where $\{e_1, e_2\}$ is, in any case, the natural basis having product as given in Lemma \ref{nerja}.
We analyze these systems (having 16 equations with 15 variables) for the different $\mathfrak{D}_i's$.
\medskip

$\mathfrak{D}_1$. Since the algebra $A_1$ is both commutative and associative, the  general solution of the system ($\ref{mistela}$) is $\sum_1^{15}\lambda_i=0$. Hence, a basis for the space of solutions is:  
    \begin{enumerate}
        \item $\lambda_1=1$, $\lambda_2=-1$, $\lambda_i=0$  ($i\ne 1,2$).
        \item $\lambda_1=1$, $\lambda_3=-1$, $\lambda_i=0$ ($i\ne 1,3$).
        \item $\lambda_1=1$, $\lambda_j=-1$, $\lambda_k=0$ ($k\ne 1,j$).
        \item $\lambda_1=1$, $\lambda_{15}=-1$, $\lambda_i=0$ ($i\ne 1,15$).
    \end{enumerate}
    These fourteen identities are trivial in the presence of commutativity and associativity.
\medskip

$\mathfrak{D}_2$. The system (\ref{mistela}) gives the solutions
    $$\begin{matrix}
     \lambda_3=-\lambda_1-\lambda_2, & \lambda_{10}=-\lambda_5-\lambda_9,\cr
     \lambda_{13}=-\lambda_4-\lambda_8, & \lambda_{14}=-\lambda_6-\lambda_{12},\cr
     \lambda_{15}=-\lambda_7-\lambda_{11}. & 
     \end{matrix}$$
     
     This space of solutions is $10$-dimensional but in the presence of commutativity, there are only two identities for the algebras having square $\mathfrak{D}_2$:
     
     \begin{equation}\label{anis}
     \begin{cases}
     (x y) (z t)=(x t) (y z),\cr
     (y (x z)) t=(y (t z)) x.
     \end{cases}
     \end{equation}
     $\mathfrak{D}_3$, $\mathfrak{D}_4$, $\mathfrak{D}_5$. For any of these algebras we find the same set of solutions for 
     the system (\ref{mistela}):
   
     \begin{tabular}{ll}
$\lambda _{13}=-\lambda
   _{14}-\lambda _{15}$, &
   $\lambda _{10}=-\lambda
   _{11}-\lambda _{12}$,\cr
   $\lambda _8=-\lambda _9+\lambda
   _{11}+\lambda _{15}$, &
   $\lambda _7=-\lambda
   _{11}-\lambda _{15}$,\cr
   $\lambda _6=-\lambda
   _{12}-\lambda _{14}$, &
   $\lambda _5=-\lambda _9+\lambda
   _{11}+\lambda _{12}$,\cr
   $\lambda _4=\lambda _9-\lambda
   _{11}+\lambda _{14}$, &
   $\lambda _3=-\lambda _9-\lambda
   _{14}$,\cr
   $\lambda _2=\lambda
   _9-\lambda _{11}-\lambda
   _{12}-\lambda _{15}$, &
   $\lambda _1=\lambda
   _{11}+\lambda _{12}+\lambda
   _{14}+\lambda _{15}.$
\end{tabular}

Thus the identities we found in the algebras of type $\mathfrak{D}_i$ with $i=3,4,5$, are:
\begin{equation}\label{delmono}
\begin{cases}
(x z)(y t) + (y,z,x) t +(x,t,y) z = (x t) (y z), \cr
(x y)(z t) + (y,z,t) x + (x,z,y) t + (t,y,x) z = (x z) (y t), 
\end{cases}
\end{equation}
\noindent where the associator $(x,y,z)$ is, as usual, $(x,y,z):=(xy)z-x(yz)$.

Summarizing, and checking (it is straightforward) that
\eqref{delmono} are verified by any two dimensional evolution algebra (even in the nonperfect case), we get the following.

\begin{proposition}\label{IdenGradoCuatro}
The degree four identities satisfied by the perfect two dimensional evolution algebras are the following: $A_1$ satisfies any such identity; $A_{2, \a}$ satisfies the identities \eqref{anis}; $A_{3, \a}$, $A_{4, \a}$ and $A_{5, \a, \b}$ satisfy the identities \eqref{delmono}.
Any perfect two dimensional evolution algebra satisfies the identities \eqref{delmono}.
\end{proposition}
\medskip

\subsection{Identities for non-perfect two dimensional evolution algebras} 
\hskip 5.cm

\noindent
Identities on $A_5$. It is not difficult to check that this algebra does not satisfy nontrivial identities of degree $3$. However, since 
$(A_5^2)^2=0$ (see Proposition \ref{cuboNonulo1}), we deduce 
that the following identity holds 
\begin{equation}\label{castan1}
(xy)(zt)=0 \text{ for any } x,y,z,t.
\end{equation}
In order to find other identities of degree $4$ we define again an element
$$w(x,y,z,t)=\lambda _1 (x y) (z t)+\lambda _2 (x
   z) (y t)+\lambda _3 (x t) (y
   z)+\lambda _4 (x (y z)) t+\lambda _5
   (y (x z)) t+$$
   $$\lambda _6 (z (y
   x)) t+\lambda _7 (t (x y)) z+\lambda
   _8 (x (t y)) z+\lambda _9 (y (x
   t)) z+\lambda _{10} (y (t z))
   x+$$
   $$\lambda _{11} (t (y z)) x+\lambda _{12}
       (z (t y)) x+\lambda _{13} (x (t
   z)) y+\lambda _{14} (z (t x))
   y+\lambda _{15} (t (x z)) y,$$
and solve the linear system 
\begin{equation}\label{mistelaDos}
    w(e_i,e_j,e_k,e_l)=0
\end{equation} 
for $i,j,k,l\in\{1,2\}$, considering of course the basis $\{e_1,e_2\}$ with $e_1^2=e_1-e_2$ and $e_2^2=-e_1+e_2$. We get the identities:
\begin{equation}
\begin{cases}\label{castan2}
(t (x y)) z = (z (y x)) t,\cr
(x(y z))
   t+(z (t x))
   y =(t (x z))
   y+(z (t y)) x,\cr
   (x (t z))
   y+(z (y x)) t= (t (x z))
   y+(z (t y))
   x,\cr
   (t (x y)) z+(x
   (t z)) y=(t
   (x z)) y+(z
   (t y)) x,\cr
   (t (y z))
   x+(z (t x))
   y=(t (x z))
   y+(z (t y)) x.
   \end{cases}
   \end{equation}

Identities for $A_6$ and $A_7$. Both algebras are associative. In the case of $A_6$ we have 
$A_6^3=0$ hence any product of three elements is zero. Thus, associativity holds trivially.
For $A_7$ we have $A_7\cong \K\times\K$ with product given by $(x,y)(x',y')=(xx',0)$. So, associativity is also trivial since it holds in $\K$.

Identities for $A_{8,\a}$ ($\a\ne 0$). It is not difficult to see that this algebra does not satisfy any identity of degree three.
Proceeding as in the first item we find the identities:
\begin{equation}\label{castan3}
    \begin{cases}
(x(yz))t=(t(yz))x,\cr 
(t (y z)) x+(x y) (z t)+(z (t x)) y=
(t (x y))z+(x t)(yz) +(x (t z)) y.
\end{cases}
\end{equation}

Collecting the computations in the non-perfect case and checking that any nonperfect two dimensional evolution algebras satisfies the identities in \eqref{delmono} (which is not difficult) we have the following.

\begin{proposition}
The unique non-perfect two dimensional evolution algebras satisfying a nontrivial identity of degree $3$ are $A_6$ and $A_7$, which are associative algebras. As for identities of degree $4$, the algebra $A_5$ satisfies \eqref{castan1} and \eqref{castan2}. The algebras $A_6$ and $A_7$ satisfy any identity of degree $4$ (begin associative and commutative). The algebra $A_{8,\a}$
satisfies \eqref{castan3}. Any (perfect or nonperfect) two dimensional evolution algebra satisfies the identities in \eqref{delmono}.
\end{proposition}

\begin{remark}\rm
 In fact, it follows also from a direct computation that any commutative algebra of dimension $2$ satisfies the equations in \eqref{delmono}.
\end{remark}

%%%%%%End of non-perfect case %%%%%%%%%%%%%%%%%%%%%%%%
%%%%%%%%%%%%%%%%%%%%%%%%%%%%%%
%%%%%%%%%%%%%%%%%%%%%%%%%%%%%%
%%%%%%%%%%%%%%%%%%

\section{Automorphisms and derivations of two dimensional evolution algebras}

We study the automorphisms as an algebraic group of the two dimensional evolution algebras. This allows to obtain more information than when considering its structure as a group. 
In the forthcoming paper \cite{CGMMS} we determine the Hopf algebra representing the affine group scheme of the automorphism group.

\subsection{Automorphism of evolution algebras.}

We start by computing the automorphisms of the two dimensional evolution algebras. Taking into account  \cite[Theorem 2.2]{CSV} the procedure is as follows: we need to determine those invertible matrices $P\in {\mathcal M}_2(\K)$ such that
$$M_B=P^{-1}M_BP^{(2)} \quad \text{and}\quad  M_B(P\ast P)=M_B \begin{pmatrix}
p_{11} & p_{12} \\ p_{21} & p_{22}
\end{pmatrix} = 0.$$

Notice that for perfect evolution algebras the equation $M_B(P\ast P)=0$ implies that $p_{11} p_{12}=0$ and $p_{22}p_{21}=0$.

{\bf Automorphisms of $A_1$}. It is easy to see that there are only two.  
$$\aut(A_1)=\left\{\begin{pmatrix}
1 & 0 \\ 0 & 1
\end{pmatrix}, \quad \begin{pmatrix}
0 & 1 \\ 1 & 0
\end{pmatrix}\right\},$$
which is isomorphic to $(\Z_2, +)$. 
\medskip

{\bf Automorphisms of $A_{2, \alpha}$}. 
Denote by $S_\beta$ the set of roots of the polynomial $x^3-\beta$. In this case the system we have to solve is:

$$\begin{pmatrix}
p_{11} & p_{12} \\
p_{21} & p_{22}
\end{pmatrix}
\begin{pmatrix}
0 & \alpha \\
1 & 0
\end{pmatrix} = 
\begin{pmatrix}
0 & \alpha \\
1 & 0
\end{pmatrix}
\begin{pmatrix}
p_{11}^2 & p_{12}^2 \\
p_{21}^2 & p_{22}^2
\end{pmatrix}; \quad \alpha p_{21}p_{22}=0; \quad p_{11}p_{12}=0.
$$

From here we consider the cases $p_{11}=0$ or $p_{11}\neq 0$. In the first one we get the automorphism $\begin{pmatrix}
0 & \omega\tau \\ \omega^{-1}\tau^{-1} & 0
\end{pmatrix}$, where  $\omega \in S_1$ and $\tau \in S_\alpha$. In the second case we need to distinguish weather or not $x^3-\alpha$ has a root and, in each of these cases, we also have to consider when $x^2+x+1$ has a root or when it has not. Summarizing all the computations, that we do not include, we get that the group of automorphisms is:

$$\aut(A_{2,\alpha})=\left\{\begin{pmatrix}
0 & \omega\tau \\ \omega^{-1}\tau^{-1} & 0
\end{pmatrix} \ \vert \ \omega \in S_1, \tau \in S_\alpha\right\} 
\sqcup 
\left\{\begin{pmatrix}
\omega & 0 \\ 0 & \omega^2
\end{pmatrix} \ \vert \ \omega \in S_1\right\}.$$

A detailed description of this group, depending on $\K$ follows. Indeed, if ${\rm char}(\K)=3$, then $S_1=\{1\}$ and: 
$S_\alpha=\emptyset$ or $\vert S_\alpha\vert =1$; therefore
$$\aut(A_{2,\a})\cong \begin{cases}
 \{1\} & \text{if}\  S_{\alpha}=\emptyset, 
\\
\\
\Z_2 & \text{ otherwise.}
\end{cases}$$

If ${\rm char}(\K)\neq 3$, then $\vert S_1 \vert=1$ or 3 and $\vert S_\a\vert= 0, 1$ or 3. Taking into account that $\vert S_1\vert = 1$ implies $\vert S_\a\vert = 0$ or 1 and that $\vert S_1\vert = 3$ implies $\vert S_\a\vert = 0$ or 3, we have the four cases detailed below. 

$$\aut(A_{2,\a})\cong \begin{cases}
 \left\{1\right\} & \text{when}\ \vert S_1\vert = 1 \ \text{and}\ \vert S_\a\vert = 0,
\\
\\
\Z_2  & \text{when}\ \vert S_1\vert = 1 \ \text{and}\ \vert S_\a\vert = 1,
\\
\\
\Z_3 & \text{when}\ \vert S_1\vert = 3 \ \text{and}\ \vert S_\a\vert = 0,
\\
\\
S_3 & \text{when}\ \vert S_1\vert = 3 \ \text{and}\ \vert S_\a\vert = 3,
\end{cases}$$
where $S_3$ denotes the permutation group of three elements.

\begin{remark}
\rm
When the characteristic of $\K$ is 3 then the algebraic group of ${\rm Aut}(A_{2,\alpha})$ is zero dimensional while its Lie algebra is one dimensional (see Proposition~\ref{Casillo}). 
\end{remark}

For the remaining cases we will not include the computations. The algebraic group for the upcoming perfect algebras  is a finite constant group since the corresponding Hopf algebra is \'etale. We specify the groups of $\K$-points in each case.
\medskip

{\bf Automorphisms of $A_{3,\alpha}$ and $A_{4, \alpha}$}. After computing we have:

$$\aut(A_{3,\alpha})\cong \{1\}, \quad 
\aut(A_{4,\alpha})\cong \{1\}.$$

{\bf Automorphisms of $A_{5,\alpha, \beta}$}. In this case, we get

$$\aut(A_{5,\alpha, \beta}) \cong\begin{cases}
 \left\{\begin{pmatrix}
1 & 0 \\ 0 & 1
\end{pmatrix}, \begin{pmatrix}
0 & 1 \\ 1 & 0
\end{pmatrix} 
\right\} &  \text{if} \ \alpha=\beta,
\\
\\
\{1\} & \text{if} \ \alpha\neq \beta.
\end{cases}$$

{\bf Automorphisms of $A_{5}$}. After computing we get:

$$
\aut(A_{5}) \cong 
\begin{cases}
 \left\{
\begin{pmatrix}
a & 1-a \\ 1-a & a
\end{pmatrix} \ \vert \ a \in \K\right\} & \text{if}\ \text{char}(\K) =2,
\\
\\
\left\{\begin{pmatrix}
a & 1-a \\ 1-a & a
\end{pmatrix} \ \vert \ a \in \K \setminus{\left\{\frac{1}{2}\right\}}
\right\} & \text{if}\  \text{char}(\K) \neq 2.
\end{cases}
$$

In the first case, the group is isomorphic to $(\K, \cdot)$, where the product $\cdot$ is given by $x \cdot y= 1+x+y$. Moreover, this group is isomorphic to $(\K, +)$ via the isomorphism given by $x \mapsto 1+x$.

In the second case, the group is isomorphic  to $(\K, \cdot)$, where the product $\cdot$ is given by $x \cdot y= 2xy-x-y+1$. Moreover, this group is isomorphic to $(\K^\times, .)$, where $.$ is the product in the field $\K$. An isomorphism from this group into $(\K, \cdot)$ is given by $x \mapsto \frac{x}{2}+ \frac{1}{2}$.

Summarizing:
$$ \aut(A_5)\cong
\begin{cases}
 (\K, +) \quad \text{if}\quad \text{char}(\K) = 2, \\
 (\K^\times, .) \quad \text{if}\quad \text{char}(\K) \neq 2.
\end{cases}
$$

{\bf Automorphisms of $A_{6}$}. In this case, we get that
$$\aut(A_{6})=\left\{\begin{pmatrix}
a^2 & b\\ 0& a
\end{pmatrix} \ \vert \ a \in \K^\times, b \in \K
\right\},$$
which is isomorphic to the affine group $\text{Aff}_1(\K):=
\left\{
\begin{pmatrix}
1 & 0\\ b& a
\end{pmatrix} \quad \vert \quad a\in \K^\times, \quad b \in \K
\right\}$ via the isomorphism given by:
$$\begin{pmatrix}
a^2 & b\\ 0& a
\end{pmatrix}
\mapsto%
\begin{pmatrix}
1 & 0\\ \frac{b}{a}& a
\end{pmatrix}.$$

{\bf Automorphisms of $A_{7}$}. After computing we obtain that 
$$\aut(A_{7})=\left\{\begin{pmatrix}
1 & 0\\ 0& a
\end{pmatrix} \ \vert \ a \in \K^{\times}
\right\},$$
which is isomorphic to $(\K^\times, .)$, where the dot denotes the product in the field.

{\bf Automorphisms of $A_{8,\alpha}$}. In this case, we have that 
$$\aut(A_{8})=\left\{\begin{pmatrix}
1 & 0\\ 0& 1
\end{pmatrix}, 
\begin{pmatrix}
1 & 0\\ 0& -1
\end{pmatrix}
\right\},$$
which is isomorphic to $(\Z_2, +)$ when the characteristic of $\K$ is different from 2 and isomorphic to the trivial group when the characteristic of $\K$ is 2.

\begin{example}\label{daniel}
\rm
One question that arises naturally is if two evolution algebras in the same family which have the same automorphism group need to be isomorphic. The answer is no and one example is the following: consider the field $\K=\Q(x, y)$. Then $A_{2,x}$ and $A_{2, y}$ are not isomorphic because $\overline x \notin \{\overline y, {\overline y}^2\}$, where classes are considered in $\K^\times/(\K^{\times})^3$. However, both evolution algebras have the same automorphism group, which is the trivial one.
\end{example}

\subsection{Derivations of two dimensional evolution algebras}

We start with the study of the derivations of perfect algebras.

Let $A$ be a two dimensional evolution algebra with natural basis $\{e_1, e_2\}$ and structure matrix given by $M_B=\begin{pmatrix}
\a & \b \\ \gamma & \delta
\end{pmatrix}.$
Assume $d:A \to A$ is a derivation. Then there exist $a, b, x, y \in \K$ such that
$\begin{cases}
 d(e_1)= ae_1 + be_2,\\
 d(e_2)=xe_1 +ye_2.
\end{cases}$

Notice that a map $d$ is a derivation if and only if $d$ is a linear map which satisfies $d(e_i^2)=2d(e_i)e_i$, for $i=1, 2$, and $0=d(e_1e_2)=d(e_1)e_2+e_1d(e_2)$. These three identities give the following equations in the indeterminates $a, b, x, y$. 

$$(S) \equiv \ \begin{cases}
 (1) \quad a\a-x\beta =0,\\
 (2) \quad y\delta-b\gamma=0,\\
 (3) \quad b\gamma+x\alpha = 0, \\
 (4) \quad b\delta + x\beta =0, \\
 (5) \quad 2a\b=b\a + y \b,\\
 (6) \quad 2y\gamma = a \gamma+x\delta.\\
\end{cases}$$
Denote by $Q$ the matrix of the linear system given by equations (1) to (4), that is,

$$Q=\begin{pmatrix}
\alpha & 0 & -\beta & 0 \\
0 & \gamma & 0 & -\delta \\
0 & \gamma & \alpha & 0 \\
0 & \delta & \beta & 0
\end{pmatrix}.$$
Since $\left|\begin{matrix}
\gamma & \alpha \\
\delta & \beta 
\end{matrix}\right|=-\left| M_B\right|\neq 0$, it follows that $b=0$ and $x=0$. On the other hand, note that
$\left| Q\right|=\alpha\delta\left| M_B\right|$. Now we analyze the different cases for which this determinant is zero.

{\bf Case 1:} Assume $\alpha = 0$. 

\noindent

\noindent
Then $(5)$ and (6) gives $2a\beta = y\beta$ and $2y\gamma = a\gamma$. Applying that $\beta, \gamma$ must be nonzero (since $\vert M_B \vert \neq 0$) we obtain $y=-a$. Substituting $y$ by $-a$ in $(S)$, the system reads as follows:
$$\begin{cases}
-\delta a = 0,\\
2a\beta=y\beta, \\
-2a\gamma = a\gamma .
\end{cases}
$$
Taking into account that  $\beta, \gamma \neq 0$ we get

$$\begin{cases}
-\delta a = 0,\\
2a=y, \\
-2a = a.
\end{cases}
$$
When $\delta \neq 0$ the derivation $d$ is trivial ($d=0$). When $\delta =0$ we obtain $2a=y$ and $3a=0$. Therefore, if 
$\text{char}(\K)\neq 3$ then $d=0$ and  $\text{char}(\K)=3$ then the map $d: A \to A$ given by
$$d(re_1+se_2)= are_1-ase_2$$
is a derivation of $A$.

{\bf Case 2:} Assume $\alpha \neq 0$. 

\noindent

\noindent
If $\delta \neq 0$ then $\vert Q \vert \neq 0$, which implies $d=0$. Otherwise, taking into account that when $\delta =0$ then $\beta, \gamma\neq 0$, because $\vert M_B\vert \neq 0$, the only solution to  (S) is the trivial one. Hence $d=0$. 

We summarize the results in the following proposition.

\begin{proposition}\label{Casillo}
Let $A$ be a perfect two dimensional evolution $\K$-algebra. Then $\der(A)=0$ except when $\text{char}(\K)=3.$ In this case,  
$$\begin{cases}
 \der(A_{2,\a})=\left\{d_a:A_{2,\a} \to A_{2,\a} \ \vert \ 
d(re_1+se_2)= are_1-ase_2,\ \text{where}\ a\in \K\right\}.\\
\der(A)=0 \ \text{whenever} \ A \neq  A_{2,\a}.
\end{cases}
$$
\end{proposition}

\medskip

Now, we continue with the derivations of the non-perfect evolution algebras.

{\bf Derivations of $A_{5}$}.

$$\der(A_5):=
\left\{
\begin{pmatrix}
a & -a\\ -a& a
\end{pmatrix} \quad \vert \quad a \in \K
\right\},$$
which is a one  dimensional $\K$-algebra.

{\bf Derivations of $A_{6}$}.

$$\der(A_6):=
\left\{
\begin{pmatrix}
2a & b\\ 0& a
\end{pmatrix} \quad \vert \quad a, b \in \K
\right\},$$
which is a two dimensional $\K$-algebra. It is the only two dimensional non abelian Lie algebra; it
%; denote it by $\mathcal L$. It 
 has a basis $\{u,v\}$ with $[u,v]=u$.

{\bf Derivations of $A_{7}$}.
$$\der(A_7):=
\left\{
\begin{pmatrix}
0 & 0\\ 0& a
\end{pmatrix} \quad \vert \quad a \in \K
\right\},$$
which is a one  dimensional $\K$-algebra.

{\bf Derivations of $A_{8,\alpha}$}.

The only derivation is the trivial one.

\medskip

We summarize all the results of the section in the following table.

{\tiny \begin{center}
\begin{table}[h!]
{\renewcommand{\arraystretch}{1.3}    
\begin{tabular}{|c|ll|c|}
\hline
Algebra & Group of Automorphisms & & Derivations\cr
\hline\hline

$A_1$ & $\Z_2$ &  & 0 \cr 
\hline
$A_{2,\a}$ & $ \begin{array}{l}
\left\{\begin{pmatrix}
0 & \omega\tau \\ \omega^{-1}\tau^{-1} & 0
\end{pmatrix} \ \vert \ \omega \in S_1, \tau \in S_\alpha\right\} 
\sqcup 
\left\{\begin{pmatrix}
\omega & 0 \\ 0 & \omega^2
\end{pmatrix} \ \vert \ \omega \in S_1\right\}, \\
\text{where} \ S_\beta \text{ is the set of roots of the polynomial } x^3-\beta.
\end{array}$
&  &

$\begin{array}{cl}
\left\{\begin{pmatrix}
a & 0 \\ 0 &-a
\end{pmatrix} \ \vert \ a \in \K\right\} & \text{if}\ {\text{char}(\K)=3,}\\

0  & \text{if}\ {\text{char}(\K) \neq 3}

 \end{array}
$
\cr 
\hline
$A_{3,\a}$ & Trivial Group &  &
0
\cr 
\hline
$A_{4,\a}$ & Trivial Group  &  & 0
\cr 
\hline
$A_{5,\a,\b}$ & $\aut{(A_{5,\a,\b})}\cong\begin{cases}
 \Z_2, \quad \text{if}\quad \alpha = \beta,  \\
 \text{Trivial Group},  \quad \text{if}\quad \alpha \neq \beta. 
\end{cases}$ & &
0 \cr 
\hline
$A_{5}$ & $ 
 \begin{cases}
 (\K, +) \quad \text{if}\quad \text{char}(\K) = 2,\\
 (\K^\times, .) \quad \text{if}\quad \text{char}(\K) \neq 2 .
\end{cases}
$ &   & $\left\{
\begin{pmatrix}
a & -a\\ -a& a
\end{pmatrix} \quad \vert \quad a \in \K
\right\}$ 

\cr
\hline
$A_{6}$ & 
$\left\{
\begin{pmatrix}
1 & 0\\ b& a
\end{pmatrix} \quad a\in \K^\times, \quad b \in \K
\right\}$  & &

$\left\{
\begin{pmatrix}
2a & b\\ 0& a
\end{pmatrix} \quad \vert \quad a, b \in \K
\right\}$
\cr 
\hline
$A_{7}$ &  $(\K^\times, .)$& &
$\K$
\cr 
\hline
$A_{8,\a}$ & $ \begin{cases}
 (\Z_2, +) \quad \text{if}\quad \text{char}(\K) \neq 2, \\
 \text{Trivial Group}, \quad \text{if}\quad \text{char}(\K) = 2. 
\end{cases}
$  &  &
0
\cr 
\hline
\end{tabular}
\caption{\label{automorphisms_dev} Automorphisms and derivations of two-dimensional evolution algebras.}
}

\end{table}
\end{center}}
\bigskip

\section{Associative representations}

In this section we build faithful associative representations in $\K\times \K$ of all two dimensional evolution algebras, except for the algebras $A_{3,\a}$, $A_{4,\a}$ and $A_{5,\a,\b}$ with $\a \neq \b$, for which we show that there are no faithful associative representations in $\K\times \K$. Furthermore, we describe universal faithful associative representations of all algebras, except again for  $A_{3,\a}$, $A_{4,\a}$ and $A_{5,\a,\b}$ with $\a \neq \b$, for which we show that there are no faithful universal associative representations in a commutative algebra. Before we proceed, we recall the definition of associative representations of nonassociative algebras, as defined in \cite{KS}.

By $\K  \langle x,y,*\rangle$ we will understand the free associative non-commutative algebra with involution $*$ on free generators $x, y$.

\begin{definition}
\rm
Let $A$ be an algebra and $(B,*)$ be an associative algebra with involution $*$. A $\K$-linear map $\mu:  A \rightarrow B$ is called an \emph{associative representation} of $A$ in $(B,*)$ if there exists a bilinear polynomial $p \in \K  \langle x,y,*\rangle$ such that it holds $\mu(ab)=p(\mu(a),\mu(b))$ for all $a,b \in A$. If $\text{Ker}(\mu)=\{0\}$ then $\mu$ is called  \emph{faithful}.
\end{definition}

\begin{remark}
\rm
The definition of associative representations in \cite{KS} is done for fields whose characteristic is not 2 or 3. Therefore, for the remainder of this section, we only work with algebras over fields whose characteristic is not 2 or 3.
\end{remark}

Following \cite{KS}, given an evolution algebra $A$ and a bilinear polynomial $p(x, y)\in \K\langle x, y, \ast \rangle$ there exists an associative with involution algebra, denoted by $(\U_p(A), \ast)$, satisfying \cite[Theorem 1]{KS}. This algebra will be called the \emph{universal associative algebra} of $A$ and the polynomial $p$.
The authors in  \cite{KS} prove the existence of this algebra using tensor products (see the proof of  \cite[Theorem 1]{KS}). For evolution algebras, the universal associative algebra can be described as follows.

\begin{theorem}\label{Chicle} Let $A$ be a two dimensional evolution algebra with basis $\{e_1, e_2\}$ and structure matrix relative to this bases $(\omega_{ij})$; let $p(x, y) \in \K\langle x, y, \ast \rangle$. The universal associative algebra of $A$ and $p$ is $\K\langle x, y, \ast \rangle/I$, where $I$ is the $\ast$-ideal generated by the polynomials
$p(x, x) - \omega_{11}x-\omega_{21}y, \ p(y, y) - \omega_{12}x-\omega_{22}y, \ p(x, y), \ \text{and } p(y, x)$, and $\mu: A \to \K\langle x, y, \ast \rangle/I$ is given by $\mu(e_1)=\overline{x}, \mu(e_2)=\overline{y}$.
\end{theorem} 
\begin{proof}

We have to show that conditions (U1) and (U2) below are satisfied.

(U1) For $\mu: A \to \K\langle x, y, \ast \rangle/I$ given by $\mu(e_1)=\overline{x}, \mu(e_2)=\overline{y}$, it holds:

$$\quad \mu(ab) = p (\mu(a), \mu(b)),$$
 for any $a, b\in A$. 
\medskip

(U2) For any associative algebra  with involution $(B, \ast)$ and any linear map $\psi: A \to B$ satisfying
$$\psi(ab) = p (\psi(a), \psi(b)),$$
whenever $a, b\in A$, there exists a unique homomorphism of (associative) algebras with involution $f: \K\langle x, y, \ast \rangle/I  \to B$ such that the following diagram commutes
$$
  \begin{tikzcd}
    A \arrow{r}{\mu} \arrow[swap]{dr}{\psi} & \K\langle x, y, \ast \rangle/I \arrow{d}{f} \\
     & B
  \end{tikzcd}
$$

To prove that (U1) and (U2) are satisfied notice that $$\mu(e_1^2)= \mu (\omega_{11} e_1 + \omega_{21} e_2) = \omega_{11} \mu(e_1) + \omega_{21}\mu(e_2) = \omega_{11} \overline{x} + \omega_{21} \overline{y}=p(\overline{x},\overline{x}) = p(\mu(e_1),\mu(e_1)),$$
$$\mu(e_2^2)= \mu (\omega_{12} e_1 + \omega_{22} e_2) = \omega_{12} \mu(e_1) + \omega_{22}\mu(e_2) = \omega_{12} \overline{x} + \omega_{22} \overline{y}=p(\overline{y},\overline{y}) = p(\mu(e_2),\mu(e_2)),$$
$\mu(e_1 e_2) = 0 = p(\overline{x},\overline{y}) = p(\mu(x),\mu(y)) $ and $\mu(e_2 e_1) = 0 = p(\overline{y},\overline{x}) = p(\mu(y),\mu(x)) $.
Therefore, condition (U1) holds. 

Let $(B, \ast)$ be an associative algebra with involution, and $\psi: A \to B$ satisfying 
 $\psi(ab)= p(\psi(a), \psi(b))$ for any $a, b\in A$. By the universal property of $\K\langle x, y, \ast\rangle$ there exists a $\ast$-homomorphism of associative algebras $g: \K\langle x, y, \ast\rangle\to B$ such that $g(x)=\psi(e_1)$ and $g(y)=\psi(e_2)$. We claim that $I\subseteq \Ker (g)$. To see this, let $\mu': A \rightarrow \K\langle x, y, \ast\rangle$ be the map given by $\mu'(e_1) = x$ and $\mu'(e_2)=y$. Then

$$\begin{array}{cll}
g(p(x,x) - \omega_{11}x-\omega_{21}y)    & = & g (p(x,x)) - \omega_{11}g(x) - \omega_{21}g(y)\\
  & =   & g(p(\mu'(e_1),\mu'(e_1))) - \omega_{11}g(\mu'(e_1)) - \omega_{21}g(\mu'(e_2))\\
  & = &  p(g(\mu'(e_1)),g(\mu'(e_1))) - \omega_{11}g(\mu'(e_1)) - \omega_{21}g(\mu'(e_2)) \\
  & = & p(\psi(e_1 e_1)) - \omega_{11} \psi (e_1) - \omega_{21} \psi(e_2) \\
  & = & \psi(\omega_{11}e_1+\omega_{21} e_2) - \omega_{11} \psi(e_1) - \omega_{21}\psi(e_2) = 0
\end{array}$$

and 

$$g(xy)= g(x)g(y)= \psi(e_1)\psi(e_2)=p(\psi(e_1), \psi(e_2))=\psi(e_1e_2)=0.$$

The other conditions follow similarly. Then, the map $g$ gives rise to a $\ast$-homomorphism of associative algebras with involution $f: \K\langle x, y, \ast \rangle \to B$ such that $f \circ \mu = \psi$. Hence $(U2)$ is satisfied.

\end{proof}

The map $\mu$ defined in the theorem above is called the \emph{universal associative representation} for the algebra $A$ and the polynomial $p$. 

\begin{remark}
\rm
The theorem above can be easily adapted to an evolution algebra of any dimension. We leave this to the reader.
\end{remark}

In some cases the universal associative algebra of a two dimensional evolution algebra is also commutative. This will not be the case in general. For that reason we introduce Definition \ref{rabodetoro}. Given a finite set of variables, say $X$, by $\K[X]$ we understand the free associative commutative algebra.

\begin{definition}\label{rabodetoro}
\rm
Let $A$ be an algebra and $(B,*)$ be an associative and commutative algebra with involution $*$. A $\K$-linear map $\mu:  A \rightarrow B$ is called an \emph{associative and commutative representation of} $A$ \emph{in} $(B,*)$ if there exists a bilinear polynomial $p \in \K\ang{x, y,*}$ such that $\mu(ab)=p(\mu(a),\mu(b))$ for all $a,b \in A$. If $\text{Ker}(\mu)=\{0\}$ then $\mu$ is called  \emph{faithful}.
\end{definition}

Reasoning as in \cite{KS}, given an evolution algebra $A$ and a bilinear polynomial $p(x, y)\in \K\ang{x, y,*}$ there exists an associative and commutative algebra with involution, which we denote by $\CU_p(A)$, satisfying an analogue to \cite[Theorem 1]{KS} with commutativity conditions. This algebra will be called the \emph{universal associative and commutative algebra of} $A$ \emph{and} the polynomial $p$. 

An analogue to Theorem \ref{Chicle} is valid when considering the universal associative and commutative algebra of an evolution algebra $A$ and a polynomial $p$.

Some properties concerning universal associative representations which will of use are collected.

\begin{proposition}\label{roma}
\begin{enumerate}[\rm(i)] 
\item\label{prodd} For $i=1,2$, if $A_i$ are $\K$-algebras with universal associative representation algebras $\U_i$ for the same $p\in\K\ang{x,y,*}$, then $\U_1\times \U_2$ is the universal associative representation algebra for $A_1\times A_2$ (the multiplication in the producct algebras being componentwise). In case the representations $\mu_i\colon A_i\to \U_i$ are faithful, then the canonical representation $A_1\times A_2\to\U_1\times\U_2$ given by $(a_1, a_2)\mapsto (\mu_1(a_1), \mu_2(a_2))$ is also faithful.

\item\label{prodd2} The universal associative   representation algebra of a unital commutative associative algebra $A$, relative to the product $p(xy)=xy^*$, is itself with the identity involution.

\item\label{prodd3}  The universal associative representation algebra of the based field $\K$ endowed with the zero product,  relative to the product $p(xy)=xy^*$, is $\CU_p:=\K[x,x^*]/(xx^*)$  with the involution swapping $x$ and $x^*$. 
\end{enumerate}
\end{proposition}

\begin{proof}
\eqref{prodd} follows immediately.

\eqref{prodd2}
We use as faithful representation the identity $1\colon A\to A$. If $(B,*)$ is an algebra with involution and 
$\varphi\colon A\to B$ satisfies
$\varphi(ab)=p(\varphi(a),\varphi(b))=\varphi(a)\varphi(b)^*$, then there is a unique $*$-homomorphism of associative algebras $F\colon A\to B$ such that $F1=\varphi$. In fact, we must take $F=\varphi$ which is a homomorphism of associative algebras because 
$\varphi(a)=\varphi(a)\varphi(1)^*=\varphi(1)\varphi(a)^*$ for any $a$. This implies that $\varphi(a)=\varphi(a)^*$ so that $\varphi$ is a homomorphism of associative algebras. But also $F(a^*)=F(a)=\varphi(a)=\varphi(a)^*=F(a)^*$. With the same proof we get that also $1\colon A\to A$ is the universal associative and commutative representation relative to the product $xy^*$.

\eqref{prodd3} Consider $\mu\colon\K\to\CU_p$ given by $\lambda\mapsto \lambda \bar x$. Then $\mu(ab)=\mu(0)=\bar 0$ and $\mu(a)\mu(b)^*=ab \bar x\bar{x^*}=\bar 0$. Now, if $(B,*)$ is a commutative $\K$-algebra with involution and $\varphi\colon\K\to B$ satisfies $\varphi(a)\varphi(b)^*=0$ for any $a,b\in\K$, there is a unique  $*$-homomorphism of associative algebras
$F\colon\CU_p\to B$ such that $F\mu=\varphi$: the one
induced by $F(\bar x)=\varphi(1)$ and $F(\bar x^*)=\varphi(1)^*$.
\end{proof}

\begin{corollary}\label{prodd4}  
\begin{enumerate} [\rm(i)]
\item The universal associative and commutative representation of the decomposable two dimensional evolution algebra $A_1$  relative to the product $p(x, y)=xy^\ast$ is $A_1$, which is isomorphic to $\K \times \K$ with product given by componentwise, being the involution and $\mu $ the identity map.
\item The universal associative and commutative representation of the decomposable two dimensional evolution algebra $A_7$, relative to the product $p(x, y)=xy^\ast$, is $\K\times \frac{\K[x,x^*]}{(xx^*)}$ with
$\mu(a, b)= (a, b\overline x)$, being the  involution $(a, b)^\ast=(a, b^\ast)$ and the involution in  $\K[x,x^*]/(xx^*)$ is induced by $x \mapsto x^\ast$.
\end{enumerate}
\end{corollary}

\subsection{Faithful associative representations}

We are interested in finding faithful universal representations of two dimensional evolution algebras. By means of the universal property, i.e., (U2) in Theorem \ref{Chicle} and using \cite[Theorem 2]{KS}, the
faithfulness of the universal representations can be obtained from the existence of a concrete faithful representation. In some cases, we will find particular faithful  representations in $\K \times \K$ or $M_2(\K)$.

For the faithful representations of $A_{5,\a,\b}$ with $\a \neq \b$, we will take into account the following. 

\begin{remark}\label{dino}\rm
It is well-known that any two dimensional associative $\K$-algebra $A$ with involution is isomorphic to one of the following:
\begin{enumerate}
    \item The zero product algebra.
    \item $\K\times\K$ with product $(x,y)(z,t)=(xz,0)$.
    \item $\K\times\K$ with product $(x,y)(z,t)=(0,xz)$.
    \item $\K\times\K$ with product $(x,y)(z,t)=(xz,yt)$.
    \item $\K(\epsilon)$, the dual numbers algebra, with basis $\{1,\epsilon\}$ such that $\epsilon^2=0$.
    \item A quadratic field extension $\K[x]/(p)$ where $p\in \K[x]$ is an irreducible polynomial over $\K$.
\end{enumerate}
The only possible involutions for the algebra (2) are the identity or $(x,y)^*=(x,-y)$. The only involution for the algebra (3) is the identity. For the algebra (4) the involutions are the identity and the exchange. The dual numbers algebra  has two involutions: the identity and the one given by 
$\epsilon^*=-\epsilon$ when the characteristic of $\K$ is different from 2 and the identity otherwise. Finally, the involutions in a quadratic field extension $\K[x]/(x^2+b x+c)$ are: the identity and the one satisfying $\bar x^*=-\bar x-b\bar 1$. 
\end{remark}

{\bf Faithful associative representations of $A_{1}$.}

The universal associative and commutative representation of $A_1$ relative to the product $p(x,y)=xy^*$ is given in Proposition \ref{roma} (\ref{prodd2}). To illustrate that the universal representation depends on the polynomial $p$ we present below another faithful universal representation of $A_1$ for a different $p$.

Let $B=\{e_1, e_2\}$ be a natural basis of $A_{1}$ with product given by $e_1^2=e_1, e_2^2=e_2$.
Consider the bilinear polynomial $p(x, y) = xy$ and let $I$ be the $\ast$-ideal of $\K\langle x, y, \ast \rangle$ generated by the polynomials in the set $\{x^2-x, y^2-y, xy, yx\}$.

The representation $\mu$ is faithful because $A_1$ is an associative algebra and hence the identity map $id: A_1 \to A_1$ is a faithful associative representation for the polynomial $p$. 

Finally we note that the universal algebra $\K\langle x, y, \ast\rangle/I$ is infinite dimensional as the set 
$$\{\overline{x}, \overline{x^\ast}, \overline{y}, \overline{y^\ast}, \overline{xx^\ast}, \overline{xy^\ast}, \overline{yx^\ast}, \overline{yy^\ast}, \overline{x^\ast x}, \overline{x^\ast y}, \overline{y^\ast x}, \overline{y^\ast y}, \overline{xx^\ast x}, \overline{yy^\ast y}\dots\}$$
consists of linearly independent elements.
\medskip

{\bf Associative faithful representations of $A_{2,\a}$}.

An  easy observation is that the algebra $A_{2,1}$ has a faithful associative representation $\psi \colon A_{2,1}\to {\mathcal M}_2(\K)$, where the involution in the matrix algebra is 
$$\begin{pmatrix} a & b\cr c & d\end{pmatrix}^*:=\begin{pmatrix} d & -b\cr -c & a\end{pmatrix}$$ and the product is by
$p(x, y)=x^*y^*$. The representation is given by $\psi(e_i)=E_{ii}$ (where $E_{ii}$ denotes the matrix having $1$ in place $ii$ and zero elsewhere) for $i=1,2$. 

An even simpler associative representation of $A_{2,1}$ is $\K\times\K$ with (associative) product given by componentwise multiplication and exchange involution, say $*$; in this case $p(x, y)=x^*y^*$.

Now, for the algebra $A_{2,\a}$, if 
$\root 3 \of\a\in \K$ there exists an isomorphism
$A_{2,1}\cong A_{2,\alpha}$. Indeed if $\{e_1,e_2\}$ is a natural basis of $A_{2,1}$ satisfying $e_1^2=e_2$ and $e_2^2=e_1$, and $\{f_1,f_2\}$ is a natural basis of $A_{2,\a}$ with $f_1^2=f_2$ and $f_2^2=\a f_1$, then one isomorphism is determined by 
$$\begin{matrix}
e_1\mapsto \root 3 \of {\frac{1}{\a}}f_1,\cr e_2\mapsto \root 3 \of {\frac{1}{\a^2}}f_2.
\end{matrix}$$
In case $\root 3\of{\a}\notin\K$ we have a similar representation in ${\mathcal M}_2(\K(\root 3\of\a))$ and also the simpler associative representation algebra in 
$\K(\root 3\of\a)\times \K(\root 3\of\a)$ with exchange involution and product as before. This guarantees that the universal associative representation $\mu$ is faithful.

Next we characterize the universal algebra $\K\langle x, y, \ast\rangle/I$. Recall that for $A_{2,1}$ the ideal $I$ is the $\ast$-ideal of $\K \langle x, y, \ast \rangle$ generated by the polynomials ${x^*}^2-y, {y^*}^2-x, x^*y^*, \text{ and } y^*x^*$.

Notice that $\K\langle x, y, \ast\rangle/I$ is isomorphic to $\U_p:= \K[ x, y ] / J$, where $J$ is the ideal generated by the polynomials $y^4-y, x^4-x, xy, \text{ and } yx$. An isomorphism is given by the map $\phi:\K\langle x, y, \ast\rangle/I \rightarrow \U_p $ such that $\phi(\overline{x})=\overline{x}$, $\phi(\overline{y})=\overline{y}$, $\phi(\overline{x^*})=\overline{y}^2$, and $\phi(\overline{y^*})=\overline{x}^2$. Now, $$\U_p = \text{span} \{\overline{x},\overline{x}^2, \overline{x}^3,\overline{y}, \overline{y}^2, \overline{y}^3\}.$$
We prove that $ \{\overline{x},\overline{x}^2, \overline{x}^3,\overline{y}, \overline{y}^2, \overline{y}^3\}$ is a linearly independent set. It is enough to check that $ \{\overline{x},\overline{x}^2, \overline{x}^3 \}$ and $\{\overline{y}, \overline{y}^2, \overline{y}^3\}$ are linearly independent sets. Suppose that $\alpha_1 \overline{x} + \alpha_2 \overline{x}^2 +\alpha_3 \overline{x}^3 = 0$. Then $\alpha_1 x + \alpha_2 x^2 +\alpha_3 x^3$ belongs to the ideal $J$ and hence it must be in the ideal generated by $x^4-x$. By a degree argument we see that $\alpha_1=\alpha_2=\alpha_3=0$ as desired. The proof to see that $\{\overline{y}, \overline{y}^2, \overline{y}^3\}$ is a linearly independent set is similar. Therefore, $$\U_p = \text{span} \{\overline{x},\overline{x}^2, \overline{x}^3\}\oplus \text{span} \{\overline{y}, \overline{y}^2, \overline{y}^3\} = \frac{x \K[ x ]}{ \langle x^4 -x\rangle } \oplus \frac{ y \K [y]}{\ang{y^4 -y}} .$$

As a final remark, notice that $u=\frac{1}{3} (\overline{x} + \overline {x}^2+\overline{x}^3)$ is an idempotent in $J_1:=\displaystyle \frac{x \K[ x ]}{\ang{x^4 -x}}$. Also, the element $w=\frac{1}{3} (\overline{y} + \overline {y}^2+\overline{y}^3)$  is an idempotent in $J_2:=\displaystyle \frac{y \K[ y ]}{\ang{y^4 -y}}$. Furthermore, the quotient of  $J_1$ by the ideal $J_1'$ generated by $\overline{x}-\overline{x}^3$ and $\overline{x}^2-\overline{x}^3$ is isomorphic to $\K u$ and the quotient of  $J_2$ by the ideal $J_2'$ generated by $\overline{y}-\overline{y}^3$ and $\overline{y}^2-\overline{y}^3$ is isomorphic to $\K w$. Therefore, the faithful representation of $A_{2,1}$ in $\K \times \K$ which we described above can be obtained as the quotient $\U_p/ (J_1' \oplus J_2')$.

\medskip

{\bf Associative faithful representations of $A_{3,\a}$}.

There is no faithful commutative universal representation. See the Appendix for a proof using the software Magma \cite{Magma}.

\medskip

{\bf Associative faithful representations of $A_{4,\a}$}.

There is no faithful commutative universal representation. For a proof of this fact  using the software Magma \cite{Magma}, see the Appendix.

{\bf Associative faithful representations of $A_{5,\a,\b}$, with $0 \neq \a\b\neq 1$}.
\medskip

{\bf Case 1: $\a = \b.$} 

For the algebra $A_{5,\a,\a}$ we find a faithful associative representation $\psi\colon A_{5,\a,\a}\to \K \times \K$, with (associative) product given by $(a,b)(c,d)=(ac,bd)$ and the exchange involution $(a,b)^*=(b,a)$. The nonassociative product is $p(x,y)=xy+\a x^*y^*$. The representation is given by  $$\psi(e_1)=(0,1)  \quad \textrm{and} \quad \psi(e_2)=(1,0).$$

Next we characterize the universal algebra $\K\langle x, y, \ast\rangle/I$ for $p$ the polynomial given above. Recall that the ideal $I$ is the $\ast$-ideal of $\K \langle x, y, \ast \rangle$ generated by the polynomials $x^2+\a {x^*}^2-x-\a y, y^2+\a {y^*}^2- \a x -y, x y+\a x^*y^* \text{ and } yx +\a y^*x^*$. 
 Notice that  $$(xy,x^*y^*) \begin{pmatrix}
 1 & \a \\ \a & 1
 \end{pmatrix} \in I\times I,$$
 hence $xy, x^*y^*,y^*x^* \text{ and } yx$ belong to $I$. 
 Since ${x^*}^2\equiv \a^{-1}(x+\a y-x^2)$ (modulo\ $I$) and $x$ and $y$ commute with $x$ and $y$, we conclude that ${x^*}^2$ commutes with $x$ and $y$ (modulo $I$). Furthermore, it is easy to see that ${x^*}^2$ commutes  with $x^*$ and with $y^*$. The conclusion is that ${x^*}^2$ is central (modulo $I$). In a similar way we see that ${y^*}^2$ is central (modulo $I$), so $x^2$ and $y^2$ are central and therefore $x+\a y$ and $\a x+y$ are central, what implies that $x$ and $y$ are central (modulo $I$). Thus $\U_p/I$ is a commutative algebra. A Gr\"{o}ebner basis found with Singular \cite{Singular} is 
 $\{y^2-y, x^2-x, y-x^*, x - y^*, xy, yx \}$. 
Therefore $\K\langle x, y, \ast\rangle/I$ is isomorphic to $\U_p:= \K[ x, y ] / J$, where $J$ is the ideal generated by the polynomials $y^2-y, x^2-x, xy, yx$, and the isomorphism is given by the map $\phi:\K\langle x, y, \ast\rangle/I \rightarrow \U_p $ such that $\phi(\overline{x})=\overline{x}$, $\phi(\overline{y})=\overline{y}$, $\phi(\overline{x^*})=\overline{y}$, and $\phi(\overline{y^*})=\overline{x}$. Hence, $\U_p \cong \K \oplus \K$, where the involution in $\K\oplus \K$ is the exchange evolution. Notice that we have shown that the universal representation of $A_{5,\a,\a}$ coincides with the representation $\psi$ given above.
\medskip

{\bf Case 2: $\a\neq \b.$}

There is no faithful commutative universal representation. See the Appendix for a proof using the software Magma \cite{Magma}.

 For $A_{5,\a,\b}$, with $\a\neq \b$, we prove with the aid of the computer system  Mathematica \cite{Wolfram}, that there are no faithful associative representations of $A_{5,\a,\b}$ into $\K \times \K$. To see this we consider all possible associative multiplications and involutions in $\K \times \K$ and all possible non associative products (given by a bilinear polynomial $p(x,y)$). The bilinear polynomial $p(x,y)$ is given by $$p(x,y)= \lambda_1 xy + \lambda_2 yx + \lambda_3 x y^* + \lambda_4 y^*x + \lambda_5 x^*y +\lambda_6 yx^* + \lambda_7 y^*x^* + \lambda_8 x^*y^*,$$
 which has eight parameters.
 To be able to use Mathematica \cite{Wolfram} we need to reduce the algorithmic complexity of our problem. For this, notice that if $p(x,y)$ gives rise to an associative representation then so does $q(x,y)=p(y,x)$. Furthermore, any convex combination of $p$ and $q$ also gives rise to an associative representation of the algebra. Therefore, using appropriate convex combinations, we can reduce the possible products to an expression with only four parameters, that is, we can consider the case that 
 $$p(x,y) = \lambda_1(xy+yx) + \lambda_2(xy^*+yx^*)+\lambda_3(y^*x+x^*y)+\lambda_4(x^*y^*+y^*x^*).$$
 If, furthermore, the product in $\K \times \K$ is commutative, then we can further reduce $p$ to only three parameters, that is, we only need to consider 
 $$p(x,y)= \lambda_1xy + \lambda_2(x^*y + xy^*) + \lambda_4 x^* y^*.$$
 
 If there is a $p$ as above and there exists a faithful representation $\psi: A_{5,\a,\b}\rightarrow \K \times \K$, then the following equations must be satisfied:

$$\begin{array}{lll}
  p(\psi(e_1),\psi(e_2)) & = & p(\psi(e_2),\psi(e_1))  =  0,      \\
 p(\psi(e_1),\psi(e_1)) & = & \psi (e_1) + \beta \psi(e_2),  \\
 p(\psi(e_2),\psi(e_2)) & = & \a\psi (e_1)+ \psi(e_2).    
\end{array}$$

Computing a Gr\"{o}bner basis using Mathematica \cite{Wolfram}, we have checked that the above equations give rise to an inconsistent system for all possible $p$ and all possible associative multiplications and involutions in $\K \times \K$ (described in Remark~\ref{dino}).

\begin{remark}\rm
 We have verified, using various algebraic computer systems, that there is no faithful universal representations of $A_{3,\a}$, $A_{4,\a}$, $A_{5,\a,\b}$ with $\a\neq \b$ and $p(x,y) = \lambda_1(xy+yx) + \lambda_2(xy^*+yx^*)+\lambda_3(y^*x+x^*y)+\lambda_4(x^*y^*+y^*x^*)$, with one of the coeficients $\lambda_i$ equal to zero. We have also checked that there is no universal faithful representation of these algebras when considering specific non-zero values of the parameters $\lambda_i$. Unfortunately, with all the parameters $\lambda_i$ free, the algebraic computer systems we used were not able to provide a definitive answer. 
\end{remark}

 \medskip
{\bf Associative faithful representations of $A_{5}$}.

For the algebra $A_{5}$ we find a faithful associative representation $\psi\colon A_{5}\to {\mathcal M}_2(\K)$, where the involution $T$ in the matrix algebra is the transposition of matrices,  and for  the product 
$p(x,y):=-xy-yx+x^Ty^T+y^Tx^T$. 
The representation is given by: $$\psi(e_1)=\begin{pmatrix} -1/4 & t\cr -t & -1/4\end{pmatrix} \quad \textrm{and} \quad \psi(e_2)=\begin{pmatrix} -1/4 & -t\cr t & -1/4\end{pmatrix} ,$$ 
where $t \in \K^\times$. 

A simpler faithful representation of this algebra is $\K\times\K$ with (associative) product $(x,y)(z,t)=(xz-yt, xt+yz)$ and involution $(x,y)^*=(x,-y)$.
The faithful representation is $\psi\colon A_{5}\to \K\times\K$ with 
$\psi(e_1)=(-1/4,t)$, $\psi(e_2)=(-1/4,-t)$. The nonassociative product
in this algebra is $(x,y)\cdot (z,t)=(0,-4(xt+yz))$.

For $p(x,y)=-2xy+2x^*y^*$ the universal commutative and associative algebra is
${\CU_p}:=\K[x,y,x^*,y^*] /J$, where $J$ is the $\ast$-ideal of $\K[ x,y,x^*,y^*]$ generated by the  polynomials $-2x^2+{2x^*}^2-x+y$, $-2y^2+{2y^*}^2+x-y$, and $-2xy+2x^* y^*$. Computing a Gr\"{o}bner basis we deduce that $y \equiv 2x^2+x-{2x^*}^2$ (modulo $J$) and $y^* \equiv x-y+x^*$ (modulo $J$). Therefore $${\CU_p} \cong \frac{\K[ x,x^*]}{\langle 2 x^3+2 x^2 x^*+x^2-2 x
   {x^*}^2-2 {x^*}^3-{x^*}^2\rangle}$$ endowed
   with the involution induced by $x\mapsto x^*$. The canonical monomorphism $\mu\colon A_5\to\CU_p$
   is such that $\mu(e_1)=\overline{x}$ and $\mu(e_2)= 2\overline{x}^2 + \overline{x} - 2 {\overline{x}^*}^2$.
   Observe that 
   $$2 x^3+2 x^2 x^*+x^2-2 x
   {x^*}^2-2 {x^*}^3-{x^*}^2=(x+x^*)(x-x^*)(2x+2x^*+1),$$
   hence by the Chinese Remainder Theorem $$\CU_p\cong \frac{\K[ x,x^*]}{\langle x+x^*\rangle}
   \times \frac{\K[ x,x^*]}{\langle x-x^*\rangle}
   \times \frac{\K[ x,x^*]}{\langle 2x+2x^*+1 \rangle}\cong\K[ x]\times\K[x]\times\K[x].$$

\medskip

{\bf Associative faithful representations of $A_{6}$}.

The algebra $A_{6}$ has a faithful associative representation $\psi\colon A_{6}\to {\mathcal M}_2(\K)$ for the involution in the matrix algebra given by the transposition $T$ of matrices and for the product 
$p(x,y):=-xy-yx+xy^T+y^Tx+yx^T+x^Ty-x^Ty^T-y^Tx^T$. The representation is given by: $$\psi(e_1)=\begin{pmatrix} 8t^2 & 0\cr 0 & 8t^2\end{pmatrix} \quad \textrm{and} \quad \psi(e_2)=\begin{pmatrix} z & t\cr -t & z\end{pmatrix},$$ 
where $t\in \K^\times$, $z\in\K$.

In this case, a simpler faithful associative representation is $\K\times\K$ again with (associative) product $(x,y)(z,t)=(xz-yt, xt+yz)$ and involution $(x,y)^*=(x,-y)$.
The  faithful representation is $\psi\colon A_{6}\to \K\times\K$ with 
$\psi(e_1)=(8t^2,0)$, $\psi(e_2)=(z,t)$. The nonassociative product
in this algebra is $(x,y)\cdot (z,t)=(8yt,0)$.

The universal commutative and associative algebra is ${\CU_p}:=\K[x,y,x^*,y^*] /J$, where $J$ is the $\ast$-ideal of $\K[ x,y,x^*,y^*]$ generated by the  polynomials 
$-2x^2+4xx^* - {2x^*}^2$, $-2y^2+4yy^* - {2y^*}^2-x$, and $ -2 x^*y^*+2yx^*+xy^*+2y^*x-2xy$.
Computing a Gr\"{o}bner basis we deduce that $x \equiv x^*$ (modulo $J$) and $x \equiv -2(y^*-y)^2$ (modulo $J$). Therefore ${\CU_p} \cong \K[ y,y^*]$ endowed
with the involution induced by $y\mapsto y^*$. The canonical monomorphism $\mu\colon A_6\to\CU_p$
   is such that $\mu(e_1)=-2(\overline{y}^*-\overline{y})^2$ and $\mu(e_2)= \overline{y}$.

\medskip

{\bf Associative faithful representations of $A_{7}$}.

The universal associative and commutative representation of this algebra for the product $p(x,y)=xy^*$ is given in Proposition \ref{roma} (\ref{prodd3}). 
Also the algebra $A_{7}$ has a faithful associative representation $\psi\colon A_{7}\to {\mathcal M}_2(\K)$, where the involution in the matrix algebra is once more the transposition $T$ of matrices,  and the product 
$p(x,y):=xy+yx+xy^T+y^Tx+yx^T+x^Ty+x^Ty^T+y^Tx^T$. The representation is given by: $$\psi(e_1)=\begin{pmatrix} 1/8 & 0\cr 0 & 1/8\end{pmatrix} \quad \textrm{and} \quad \psi(e_2)=\begin{pmatrix} 0 & t\cr -t & 0\end{pmatrix},$$ 
where $t \in \K^\times$.
In this case, a simpler faithful associative representation is $\K\times\K$  with (associative) product $(x,y)(z,t)=(xz-yt, xt+yz)$ and involution $(x,y)^*=(x,-y)$.
The faithful representation is $\psi\colon A_{7}\to \K\times\K$ with 
$\psi(e_1)=(1/8,0)$, $\psi(e_2)=(0,t)$. The nonassociative product
in this algebra is $(x,y)\cdot (z,t)=(8xz,0)$.

The universal commutative, associative algebra is ${\CU_p}:=\K[x,y,x^*,y^*] /J$, where $J$ is the $\ast$-ideal of $\K[ x,y,x^*,y^*]$ generated by the  polynomials $2x^2+4xx^*+{2x^*}^2-x$, $2y^2+4yy^*+{2y^*}^2$, \text{ and } $2xy+2xy^*+2yx^*+2x^*y^*$.

A Gr\"{o}bner basis for the ideal $J$ is $\left\{{y^*}^2+2 y^* y+y^2,y^* x^*+y x^*,8
   {x^*}^2-x^*,x-x^*\right\}$. 
Eliminating $x^*$ we get $ 
\left\{(y^*+y)^2,x (y^*+y),x (8
   x-1)\right\}$. Hence, defining $\omega=y^\ast + y$
 $$\frac{\K[x,y,y^*]}{\langle (y^*+y)^2, x (y^*+y), x(8x-1)) \rangle}\cong \frac{\K[x,w,y^*-y]}{\langle w^2, x w, x(8x-1)) \rangle}\cong  \left(\frac{\K[x,w]}{\langle w^2, x w, x(8x-1) \rangle}\right) [y^*-y].$$ 
 Analyzing the multiplication table of $\K[x,w]/ \langle w^2, x w, x(8x-1)\rangle$ we notice that it is isomorphic to the evolution algebra $A_7$. So, denoting $y^*-y$ by $q$, we obtain that $$\CU_p \cong A_7[q],$$ where the involution in $A_7$ is the identity and $q^*=-q$.
 The canonical monomorphism $\mu\colon A_7\to\CU_p$
   is given by $\mu(e_1)=\overline{x}$ and $\mu(e_2)= \frac{1}{2}\overline{w}-\frac{1}{2}\overline{q}$.

A (possibly) more natural associative and commutative faithful representation is the following: consider the product $p(x,y)=\frac{1}{4}(xy+x^*y+xy^*+x^*y^*)$, then define in the Laurent polynomial algebra $\K[x,x^{-1}]$ the involution such that $x^*:=x^{-1}$. Now we can 
construct
$\mu\colon A_7\to\K[x,x^{-1}]$, such that
$\mu(e_1)=1$, $\mu(e_2)=x^{-1}-x$. Then
$p(1,1)=1$, and $(x^{-1}-x)^*=x-x^{-1}$ so that
$$p(x^{-1}-x,x^{-1}-x)=\frac{1}{4}\left((x^{-1}-x)^2-
2(x^{-1}-x)^2+(x^{-1}-x)^2\right)=0,$$

$$p(1,x^{-1}-x)=\frac{1}{4}\left(
2(x^{-1}-x)-2(x^{-1}-x)\right)=0.$$
In fact, it can be proved that $\CU_p\cong\K[x,x^{-1}]$ with the above $\mu$ and $p$. 
\medskip

{\bf Associative faithful representations of $A_{8,\a}$}.
\noindent

\noindent
For this algebra we find the faithful associative representation 
$\psi\colon A_{8,\a}\to {\mathcal M}_2(\K)$ given below, for the involution in the matrix algebra given by the transposition $T$ of matrices and  product 
$$p(x,y):=k_\a (xy+yx+x^Ty^T+y^Tx^T) + h_\a (xy^T+y^Tx+yx^T+x^Ty),$$
where $k_\a=\frac{1}{8} (1-16 \alpha )$ and 
$h_\a=\frac{1}{8} (1+16 \alpha )$. The representation is given by: $$\psi(e_1)=\begin{pmatrix} 1& 0\cr 0 & 1\end{pmatrix} \quad \textrm{and} \quad \psi(e_2)=\begin{pmatrix} 0 & 1/4\cr -1/4 & 0\end{pmatrix}.$$ 
In this case, a simpler faithful associative representation (relative to the same $p$ above) is $\K\times\K$  with (associative) product $(x,y)(z,t)=(xz-yt, xt+yz)$ and involution $(x,y)^*=(x,-y)$.
The faithful  representation is $\psi\colon A_{8,\a}\to \K\times\K$ with 
$\psi(e_1)=(1,0)$, $\psi(e_2)=(0,1/4)$. The nonassociative product $p$ above, particularized to this algebra gives $(x,y)\cdot (z,t)=(16\alpha ty+xz,0)$.

Next we characterize the universal associative algebra and the commutative universal algebra associated to the polynomial above.
Let $I$ be the $\ast$-ideal of $\K \langle x, y, \ast \rangle$ generated by the polynomials $k_\a (2x^2+2{x^*}^2) + h_\a(2xx^* + 2x^* x) -x $, $k_\a (2y^2+2{y^*}^2) + h_\a(2yy^* + 2y^* y) -\a x $, and $k_\a (xy+yx+x^*y^*+y^*x^*) + h_\a (xy^*+y^*x+yx^*+x^*y)$.

Using \cite{Magma}, a (noncommutative) Gr\"{o}bner basis for the ideal $I$  is 
$$\left\{ \begin{array}{ll}

\a y y^*x^*  - \a {y^*}^2x^* - \frac{1}{8} \a x^*,\\
    (\a - \frac{1}{16}) y^2 - (\a  + \frac{1}{16}) y y^* - (\a  + \frac{1}{16}) y^*y+ (\a -
        \frac{1}{16} ){y^*}^2 + \frac{1}{4} \a x^*   ,\\
    y x^* + y^*x^*,\\
    x^*y + x^*y^*, \\
    {x^*}^2 - x^*,\\
    x - x^*.
\end{array}
\right\}
$$ 

Eliminating $x^*$ using the last two polynomials above, and computing a Gr\"{o}bner basis again, we obtain 

$$\left\{
 (\a  - \frac{1}{16}) y^2 - (\a + \frac{1}{16}) y y^* - (\a + \frac{1}{16}) y^* y + (\a -
        \frac{1}{16}) {y^*}^2 + \frac{1}{4} \a x   ,
    xy + xy^* + y x + y^*x
\right\}.$$

Eliminating $x$ in the first polynomial we get: $$\omega(y,y^*):=( y^*+y)\circ \left[ \left(\frac{1}{4 \a}-4\right)(y^2+{y^*}^2) + \left(\frac{1}{4 \a} +4\right)(yy^* + y^* y)\right].$$

Hence the universal associative algebra 
is $\K\ang{y,y^*}/(\omega(y,y^*))$ with the exchange involution. The canonical monomorphism $\mu\colon A_{8,\a}\to \K\ang{y,y^*}/(\omega(y,y^*))$ is such that $\mu(e_1)=\left(\frac{1}{4 \a}-4\right)(\overline{y}^2+{\overline{y}^*}^2) + \left(\frac{1}{4 \a} +4\right)(\overline{y}\overline{y}^* + \overline{y}^* \overline{y})$ and $\mu(e_2)=\overline{y}$.

To find the universal commutative and associative algebra we assume that commutativity holds in the above. Hence  ${\CU_p}:=\K[y,y^*] /J$, where $J$ is the $\ast$-ideal of $\K[y,y^*]$ generated by the  polynomial $\left(\frac{1}{2 a}-8\right)
   {y^*}^3+\left(\frac{3}{2
   a}+8\right) {y^*}^2
   y+\left(\frac{3}{2
   a}+8\right) y^*
   y^2+\left(\frac{1}{2
   a}-8\right) y^3$.
The involution is the exchange involution and the canonical monomorphism $\mu\colon A_{8,\a}\to \K[y,y^*] /J$ is such that $\mu(e_1)=\left(\frac{1}{4 \a}-4\right)(\overline{y}^2+{\overline{}y^*}^2) + \left(\frac{1}{2 \a} +8\right)(\overline{yy^*})$ and $\mu(e_2)=\overline{y}$.

%%%%%%%%%%%%%%%%%%%%%%%%%%%%%%
%%%%%%%%%%%%%%%%%%%%%%%%%%%%%%
%%%%%%%%%%%%%%%%%%%%%%%%%%%%%%%%%%%%%%%%%%%%%%%%
%%%%%%%%%%%%%%%%%%%%%%%%%%%%%%
%%%%%%%%%%%%%%%%%%

\newpage

\begin{center}
    
{\bf Multiplication table of two dimensional evolution algebras}
\end{center}

\medskip

In order to compile the results and terminology in the previous sections,
we include here a table containing representatives of the isomorphism classes as well as their multiplication basis relative to a natural basis. \medskip

\begin{center}
\begin{table}[h!]
{\renewcommand{\arraystretch}{1.3}    
\begin{tabular}{|c|ll|c|c|}
\hline
Algebra & Multiplication table & & Ideals & Pseudo-square\cr
\hline\hline
$A_0$ & \begin{tabular}{l}$e_1^2=0$\cr $e_2^2=0$\end{tabular} & & \begin{tabular}{l}$\langle \a e_1 +\b e_2 \rangle, $\cr $\a \in \K^{\times}$ or $\b \in \K^{\times}$ \end{tabular} &\xygraph{
!{<0cm,-0.5cm>;<1cm,-0.5cm>:<0cm,0.5cm>::}
!{(0,0)}*+{\bullet}="a"
!{(1,0)}*+{\bullet}="b"
!{(0,1)}*+{\bullet}="c"
!{(1,1)}*+{\bullet}="d"
}\cr 
\hline
$A_1$ & \begin{tabular}{l}$e_1^2=e_1$\cr $e_2^2=e_2$\end{tabular} & & $\langle e_1 \rangle, \langle e_2 \rangle$ &\xygraph{
!{<0cm,-0.5cm>;<1cm,-0.5cm>:<0cm,0.5cm>::}
!{(0,0)}*+{\bullet}="a"
!{(1,0)}*+{\bullet}="b"
!{(0,1)}*+{\bullet}="c"
!{(1,1)}*+{\bullet}="d"
"a"-"c" "b"-"d"}\cr 
\hline
$A_{2,\a}$ & \begin{tabular}{l}$e_1^2=e_2$\cr $e_2^2=\a e_1$\end{tabular} & $\a\in\K^\times$ & &
\xygraph{
!{<0cm,-0.5cm>;<1cm,-0.5cm>:<0cm,.5cm>::}
!{(0,0)}*+{\bullet}="a"
!{(1,0)}*+{\bullet}="b"
!{(0,1)}*+{\bullet}="c"
!{(1,1)}*+{\bullet}="d"
"a"-"b" "c"-"d"}
\cr 
\hline
$A_{3,\a}$ & \begin{tabular}{l}$e_1^2=e_1$\cr $e_2^2=\a e_1+e_2$\end{tabular} & $\a\in\K^\times$ & $\langle e_1 \rangle$ &
\xygraph{
!{<0cm,-.5cm>;<1cm,-0.5cm>:<0cm,.5cm>::}
!{(0,0)}*+{\bullet}="a"
!{(1,0)}*+{\bullet}="b"
!{(0,1)}*+{\bullet}="c"
!{(1,1)}*+{\bullet}="d"
 "b"-"d" "d"-"c" "a"-"c"}
\cr 
\hline
$A_{4,\a}$ & \begin{tabular}{l}$e_1^2=\a e_2$\cr $e_2^2=e_1+e_2$\end{tabular} & $\a\in\K^\times$ & &
\xygraph{
!{<0cm,-0.5cm>;<1cm,-0.5cm>:<0cm,.5cm>::}
!{(0,0)}*+{\bullet}="a"
!{(1,0)}*+{\bullet}="b"
!{(0,1)}*+{\bullet}="c"
!{(1,1)}*+{\bullet}="d"
 "a"-"b" "a"-"c" "c"-"d"}
\cr 
\hline
$A_{5,\a,\b}$ & \begin{tabular}{l}$e_1^2=e_1+\b e_2$\cr $e_2^2=\a e_1+e_2$\end{tabular} & $\a,\b\in\K^\times$, $\a\b\ne 1$ & &
\xygraph{
!{<0cm,-0.5cm>;<1cm,-0.5cm>:<0cm,.5cm>::}
!{(0,0)}*+{\bullet}="a"
!{(1,0)}*+{\bullet}="b"
!{(0,1)}*+{\bullet}="c"
!{(1,1)}*+{\bullet}="d"
 "a"-"b" "a"-"c" "c"-"d" "b"-"d"}
\cr 
\hline
$A_{5}$ & \begin{tabular}{l}$e_1^2=e_1-e_2$\cr $e_2^2=-e_1+e_2$\end{tabular} & & $\langle e_1 - e_2 \rangle$ &
\xygraph{
!{<0cm,-0.5cm>;<1cm,-0.5cm>:<0cm,.5cm>::}
!{(0,0)}*+{\bullet}="a"
!{(1,0)}*+{\bullet}="b"
!{(0,1)}*+{\bullet}="c"
!{(1,1)}*+{\bullet}="d"
 "a"-"b" "a"-"c" "c"-"d" "b"-"d"}
\cr
\hline
$A_{6}$ & \begin{tabular}{l}$e_1^2=0$\cr $e_2^2= e_1$\end{tabular} & & $\langle e_1 \rangle $ &
\xygraph{
!{<0cm,-0.5cm>;<1cm,-0.5cm>:<0cm,.5cm>::}
!{(0,0)}*+{\bullet}="a"
!{(1,0)}*+{\bullet}="b"
!{(0,1)}*+{\bullet}="c"
!{(1,1)}*+{\bullet}="d"
"c"-"d"}
\cr 
\hline
$A_{7}$ & \begin{tabular}{l}$e_1^2=e_1$\cr $e_2^2=0$\end{tabular} & & $\langle e_1 \rangle, \langle e_2 \rangle $ &
\xygraph{
!{<0cm,-0.5cm>;<1cm,-0.5cm>:<0cm,.5cm>::}
!{(0,0)}*+{\bullet}="a"
!{(1,0)}*+{\bullet}="b"
!{(0,1)}*+{\bullet}="c"
!{(1,1)}*+{\bullet}="d"
"a"-"c"}
\cr 
\hline
$A_{8,\a}$ & \begin{tabular}{l}$e_1^2=e_1$\cr $e_2^2=\a e_1$\end{tabular} & $\a\in\K^\times$ & $\langle e_1 \rangle$ &
\xygraph{
!{<0cm,-0.5cm>;<1cm,-0.5cm>:<0cm,.5cm>::}
!{(0,0)}*+{\bullet}="a"
!{(1,0)}*+{\bullet}="b"
!{(0,1)}*+{\bullet}="c"
!{(1,1)}*+{\bullet}="d"
 "a"-"c" "c"-"d"}
\cr 
\hline 
\end{tabular}
}
\caption{Two dimensional evolution algebras.}
\end{table}
\end{center}

\appendix
\bigskip

\section{}

In this appendix we include some Magma \cite{Magma} codes proving that some specific universal associative representation algebras are zero. Concretely we consider the field $Q$ of rational functions on the indeterminates $\a,\b,\lambda_i$ ($i=1,2,3,4$), with coefficients in $\Q$. Then, we consider the free associative algebra $Q\langle x,y,*\rangle$, the product $p(x,y)=\lambda_1 (xy+yx)+\lambda_2(xy^*+yx^*)+\lambda_3(y^*x+x^*y)+\lambda_4(x^*y^*+y^*x^*)$ and the $*$-ideal $I$ generated by the elements $p(x,x)-\omega_{11}x-\omega_{21} y$, $p(y,y)-\omega_{12} x-\omega_{22}y$, $p(x,y)$ and $p(y,x)$. The structure constants $\omega_{ij}$ are replaced in each case for their corresponding values.

\bigskip
 \textbf{Code for $A_{3,\a}$}
\bigskip

We can prove that there is not a commutative faithful universal representation. Indeed, computing a Gr\"oebner basis of the ideal $I+J$ where $J$ is the ideal of 
$Q\langle x,y,*\rangle$ generated by $xx^*-x^*x$ and 
$yy^*-y^*y$, we find that $x,x^*\in I+J$. We include here the Magma (\cite{Magma}) code, where $z=x^*$ and $t=y^*$:
\vspace{0.5cm}

{\small\begin{verbatim}
 Q<a,l1,l2,l3,l4> := FunctionField(Rationals(),5);
F<x,y,z,t> := FreeAlgebra(Q,4);

B := [2*l1*x*x+2*l2*x*z+2*l3*z*x+2*l4*z*z-x,
2*l1*z*z+2*l2*x*z+2*l3*z*x+2*l4*x*x-z,
2*l1*y*y+2*l2*y*t+2*l3*t*y+2*l4*t*t-a*x-y,
2*l1*t*t+2*l2*y*t+2*l3*t*y+2*l4*y*y-a*z-t,
l1*(x*y+y*x)+l2*(x*t+y*z)+l3*(t*x+z*y)+l4*(z*t+t*z),
l1*(z*t+t*z)+l2*(x*t+y*z)+l3*(t*x+z*y)+l4*(x*y+y*x)
];

D :=[x*z-z*x, y*t-t*y];
 I := ideal<F | B cat D>;
 GroebnerBasis(I);

[
    t^3 + (-l1*l2 - l1*l3 - 1/2*l1*l4 - 1/2*l2^2 - l2*l3 - 1/2*l2*l4 - 1/2*l3^2
        - 1/2*l3*l4 - 1/2*l4^2)/(l1^2*l2 + l1^2*l3 + l1*l2^2 + 2*l1*l2*l3 +
        l1*l3^2 - l2^2*l4 - 2*l2*l3*l4 - l2*l4^2 - l3^2*l4 - l3*l4^2)*t^2 +
        (-1/4*l1*l2*l4 - 1/4*l1*l3*l4 - 1/4*l1*l4^2 - 1/4*l4^3)/(l1^4*l2 +
        l1^4*l3 - l1^2*l2^3 - 3*l1^2*l2^2*l3 - 3*l1^2*l2*l3^2 - 2*l1^2*l2*l4^2 -
        l1^2*l3^3 - 2*l1^2*l3*l4^2 + 2*l1*l2^3*l4 + 6*l1*l2^2*l3*l4 +
        6*l1*l2*l3^2*l4 + 2*l1*l3^3*l4 - l2^3*l4^2 - 3*l2^2*l3*l4^2 -
        3*l2*l3^2*l4^2 + l2*l4^4 - l3^3*l4^2 + l3*l4^4)*y + (1/4*l1^2*l2 +
        1/4*l1^2*l3 + 1/4*l1^2*l4 - 1/4*l1*l2^2 - 1/2*l1*l2*l3 - 1/4*l1*l3^2 +
        1/4*l1*l4^2 + 1/4*l2^2*l4 + 1/2*l2*l3*l4 + 1/4*l3^2*l4)/(l1^4*l2 +
        l1^4*l3 - l1^2*l2^3 - 3*l1^2*l2^2*l3 - 3*l1^2*l2*l3^2 - 2*l1^2*l2*l4^2 -
        l1^2*l3^3 - 2*l1^2*l3*l4^2 + 2*l1*l2^3*l4 + 6*l1*l2^2*l3*l4 +
        6*l1*l2*l3^2*l4 + 2*l1*l3^3*l4 - l2^3*l4^2 - 3*l2^2*l3*l4^2 -
        3*l2*l3^2*l4^2 + l2*l4^4 - l3^3*l4^2 + l3*l4^4)*t,
    y^2 - t^2 - 1/2/(l1 - l4)*y + 1/2/(l1 - l4)*t,
    y*t + (l1 + l4)/(l2 + l3)*t^2 + 1/2*l4/(l1*l2 + l1*l3 - l2*l4 - l3*l4)*y -
        1/2*l1/(l1*l2 + l1*l3 - l2*l4 - l3*l4)*t,
    t*y + (l1 + l4)/(l2 + l3)*t^2 + 1/2*l4/(l1*l2 + l1*l3 - l2*l4 - l3*l4)*y -
        1/2*l1/(l1*l2 + l1*l3 - l2*l4 - l3*l4)*t,
    x,
    z
]
\end{verbatim}
}
This computation shows that $x$ and $z=x^\ast$ are in the ideal $I$, therefore the corresponding representation is not faithful. This comment can be taken into account for the rest of the section.

\medskip
 \textbf{Code for $A_{4,\a}$}
\medskip

The Magma (\cite{Magma}) code below, shows that $I$ is the whole algebra, $I=Q\langle x,y,*\rangle$ so that the universal associative algebra is zero in this case.
\vspace{0.5cm}

{\small\begin{verbatim}
Q<a,l1,l2,l3,l4> := FunctionField(Rationals(),5);
F<x,y,z,t> := FreeAlgebra(Q,4);

/* l1(xy+yx)+l2(xy*+yx*)+l3(y*x+x*y)+l4(x*y*+y*x*)*/

B := [2*l1*x*x+2*l2*x*z+2*l3*z*x+2*l4*z*z-a*y,
2*l1*z*z+2*l2*x*z+2*l3*z*x+2*l4*x*x-a*t,
2*l1*y*y+2*l2*y*t+2*l3*t*y+2*l4*t*t-x-y,
2*l1*t*t+2*l2*y*t+2*l3*t*y+2*l4*y*y-z-t,
l1*(x*y+y*x)+l2*(x*t+y*z)+l3*(t*x+z*y)+l4*(z*t+t*z),
l1*(z*t+t*z)+l2*(x*t+y*z)+l3*(t*x+z*y)+l4*(x*y+y*x)
];

 I := ideal<F | B>;
 GroebnerBasis(I);
 
 [
    x,
    y,
    z,
    t
]}


    
\end{verbatim}
\medskip

 \textbf{Code for $A_{5,\a,\b}, \a\neq \b$}
\medskip

The Magma (\cite{Magma}) code below, shows that $I$ is the whole algebra, $I=Q\langle x,y,*\rangle$ so that the universal associative algebra is zero in this case.

{\small\begin{verbatim}


Q<a,b,l1,l2,l3,l4> := FunctionField(Rationals(),6);

F<x,y,z,t> := FreeAlgebra(Q,4);

B := [2*l1*x*x+2*l2*x*z+2*l3*z*x+2*l4*z*z-x-b*y,
2*l1*z*z+2*l2*x*z+2*l3*z*x+2*l4*x*x-z-b*t,
2*l1*y*y+2*l2*y*t+2*l3*t*y+2*l4*t*t-a*x-y,
2*l1*t*t+2*l2*y*t+2*l3*t*y+2*l4*y*y-a*z-t,
l1*(x*y+y*x)+l2*(x*t+y*z)+l3*(t*x+z*y)+l4*(z*t+t*z),
l1*(z*t+t*z)+l2*(x*t+y*z)+l3*(t*x+z*y)+l4*(x*y+y*x)
];

 I := ideal<F | B>;
 
 GroebnerBasis(I);

[
    x,
    y,
    z,
    t
]
\end{verbatim}}

%%%%%%%%%%%%%%%%%%%%%%%%%%%%%%
\section*{Acknowledgments}
%%%%%%%%%%%%%%%%%%%%%%%%%%%%%%
The second author was partially supported by Conselho Nacional de Desenvolvimento Cient\'ifico e Tecnol\'ogico (CNPq) - Brasil.
The last three authors are supported by the Junta de Andaluc\'{\i}a and Fondos FEDER, jointly, through projects  FQM-336 and FQM-7156 and also by the Spanish Ministerio de Econom\'ia y Competitividad and Fondos FEDER, jointly, through project  MTM2016-76327-C3-1-P.

%%%%%%%%%%%%%%%%%%%%%%%%%%%%%%%%%%%%%%%%%%%%%%%

\end{document}